\documentclass[10pt,twocolumn]{IEEEtran}
\usepackage{amsmath, amsfonts, amssymb,color}

\newtheorem{theorem}{\indent {Theorem}}
\newtheorem{assumption}{\indent {Assumption}}

\newtheorem{lemma}{\indent {Lemma}}

\newtheorem{corollary}{\indent {Corollary}}
\def\R{\mathbb{R}}
\def\1{{\bf 1}}
\def\a{\alpha}

\def\e{{\epsilon}}

\def\bu{{\bf u}}

\def\an#1{#1}
\usepackage{epsfig}

\def\aor#1{{{#1}}}

\newenvironment{proof}[1][Proof]{\noindent\textbf{#1.} }{\ \rule{0.5em}{0.5em}}
\def\an#1{{{#1}}}

\def\re{\mathbb{R}}

\def\e{\epsilon}
\def\bx{{\mathbf x}}
\def\bw{{\mathbf w}}
\def\bz{{\mathbf z}}

\def\bg{{\mathbf g}}
\def\bv{{\mathbf v}}
\def\by{{\mathbf y}}

\begin{document}

\title{Stochastic Gradient-Push for Strongly Convex Functions on Time-Varying Directed Graphs}
\author{Angelia Nedi\'c and Alex Olshevsky\thanks{
The authors are with the Industrial and 
Enterprise Systems Engineering Department, University of Illinois at Urbana-Champaign, 104 S.~Mathews Avenue, Urbana IL, 61801, Emails: $\{$angelia,aolshev2$\}$@illinois.edu. A. Nedi\'c gratefully acknowledges the support by the NSF grants CCF 11-11342, and by the ONR Navy Basic Research Challenge N00014-12-1-0998. A preliminary version of this paper appeared in the Proceedings of the 1st IEEE Global SIP Conference, 2013.}
}
\maketitle

\begin{abstract} We investigate the convergence rate of the recently proposed subgradient-push method 
for distributed optimization over time-varying directed graphs. 
The subgradient-push method can be implemented in a distributed way without requiring knowledge of either 
the number of agents or the graph sequence; each node is only required to know its out-degree at each time. 
Our main result is a convergence rate  
of $O \left((\ln t)/t \right)$  for strongly convex  functions with Lipschitz gradients 
even if only stochastic gradient samples are available; 
this is asymptotically faster than the $O \left(( \ln t)/\sqrt{t} \right)$ 
rate previously known  for (general) convex functions.
\end{abstract}

\section{Introduction} We consider the problem of cooperatively minimizing a separable convex function by a network of nodes. 
Our motivation stems from much recent interest in distributed optimization problems which arise when large clusters of nodes (which can be
sensors, processors, autonomous vehicles or UAVs) wish to collectively optimize a global objective by means of actions taken by each node and local coordination between neighboring nodes. 

Specifically, we will study the problem of optimizing a sum of $n$ convex functions  by a network of $n$ nodes when  the $i$'th function is known only to node $i$. The functions will be assumed to be from $\R^d$ to $\R$. This problem often arises when control and signal processing algorithms are implemented in sensor networks and global agreement is needed on a parameter which minimizes a sum of local costs. Some specific scenarios
in which this problem has been considered in the literature include statistical inference \cite{rabbat}, formation control~\cite{othesis},  non-autonomous power control~\cite{ram_info}, distributed ``epidemic'' message
routing in networks~\cite{neglia}, and spectrum access coordination~\cite{li-han}.

Our focus here is on the case when the communication topology connecting the nodes is {\em time-varying} and {\em directed}. In the context of wireless networks, time-varying communication topologies arise if the nodes are mobile or if the communication between them is subject to unpredictable bouts of interference. Directed communication links are also a natural assumption as in many cases there is no reason to expect different nodes to transmit wirelessly at the same power level. 
Transmissions at different power levels will result \an{in} unidirectional communication between nodes (usually, after an initial \aor{bidirectional} exchange of ``hello'' messages). 

In our previous work \cite{AA2013} we proposed an algorithm which is guaranteed to drive all nodes to an optimal solution in this setting. Our algorithm, which we called the {\em subgradient-push}, can be implemented in a fully distributed way: no knowledge of the (time-varying) communication topology or even of the total number of nodes is required, although every node is required to know its out-degree at each time. 
The subgradient-push is a generalization of the so-called push-sum protocol for computing averages on directed graphs proposed over a decade ago~\cite{dobra-kempe} (see also the more recent development in~\cite{benezit, dominguez}). 

Our main result in~\cite{AA2013} was that the subgradient-push protocol drives all the nodes 
to an optimal solution at a rate $O \left( (\ln t)/\sqrt{t} \right)$. Here, 
we consider the effect of stronger assumptions on the individual functions. 
Our main result is that if the functions at each node are strongly convex, %and have Lipschitz gradients, 
then even if each node only has access to noisy gradients of its own function, 
an improvement to an $O \left( (\ln t)/t \right)$ rate can be achieved. 

\aor{Note that our convergence rate is quite close to best achievable rate of $O(1/t)$ in (centralized) strongly convex  optimization with noisy gradient samples of 
bounded variance \cite{NY, ABRW}.  Obtaining an algorithm with a $O(1/t)$ rate in our setting of distributed, noisy, strongly-convex optimization
over time-varying directed graphs of unknown size remains an open problem.}

Our work here contributes to the growing literature on
distributed methods for optimization over networks~\cite{AN2009, johan1, rabbat, johansson, LopesSayed2007, johan2, Lobel2011,  LobelOF2011, SN2011, nedic-broadcast,ChenSayed2012,   Duchi2012, Lu2012, gharesi2, rabbat_2013, ciblat}. 
It is a part of a recent strand of the distributed optimization literature which studies effective protocols  when interactions between nodes are unidirectional~\cite{hadj1, hadj2, hadj3, gharesi1, gharesi3, rabbat_allerton2012}. Our work is most closely related to recent developments in~\cite{rabbat_allerton2012, rabbat_cdc2012,Tsianos2011, gharesi1, Tsianos2013, ermin}. We specifically mention \cite{rabbat_allerton2012, rabbat_cdc2012}, which were the first papers to suggest the use of push-sum-like updates for optimization over directed graphs as well as~\cite{Tsianos2012, ermin} which derived O($1/t)$ convergence rates in the less stringent setting when every graph is fixed and undirected. 

Our paper is organized as follows. In Section~\ref{sec:results}, we describe the problem formally and present
the algorithm along with the main results. The results are then proved 
\an{in Sections~\ref{sec:proof1}} and~\ref{sec:proof0}. 
We conclude with some simulations in Section~\ref{sec:nums} and some concluding remarks in Section~\ref{sec:concl}.

\noindent
{\bf Notation}: 
We use boldface to distinguish between the vectors in $\R^d$ and scalars associated 
with different nodes. 
For example, the vector $\bx_i(t)$ is in boldface to identify a vector for node $i$, 
while a scalar $y_i(t)\in\R$ is not in boldface. 
The vectors such as $y(t)\in\R^n$ obtained by stacking scalar values $y_i(t)$ associated with the $n$ nodes 
are not in boldface. For a vector $y$, we will also sometimes use $[y]_j$ to denote its $j$'th entry.
For a matrix $A$, we will use $A_{ij}$ or $[A]_{ij}$ to denote its $i,j$'th entry. 
We use $\1$ to denote the vector of ones, and $\| y\|$ for the Euclidean norm of a vector~$y$. 

%%%%%%%%%%%%%%%%%%%%%%%%%%%%%%%%%%%%%%%
\section{Problem, Algorithm and Main Result\label{sec:results}} 
%%%%%%%%%%%%%%%%%%%%%%%%%%%%%%%%%%%%%%%
We consider a network of $n$ nodes which would like to collectively solve the following 
minimization problem: 
\[ \hbox{minimize } F(\bz) \triangleq \sum_{i=1}^n f_i( \bz) \quad\hbox{over $\bz\in\mathbb{R}^d$},\] 
where only node $i$ has any knowledge of the  convex function $f_i: \R^d \to \R$. 
Moreover, we assume that node $i$ has access to the 
convex function $f_i:\mathbb{R}^d\to\mathbb{R}$ only 
through the ability to generate noisy samples of its subgradient, 
i.e., given a point $\bu \in \R^d$ node $i$ can generate 
\begin{align}\label{eq:noisy-grad}
\bg_i( {\bf u}) = \aor{\nabla f_i({\bf u})} + {\bf N}_i({\bf u}),
 \end{align} 
where \aor{$\nabla f_i({\bf u})$ denotes} a subgradient of $f_i$ at ${\bf u}$ and 
${\bf N}_i(\bu)$ is an independent random vector with zero mean, i.e., 
$\mathbb{E}[{\bf N}_i({\bf u})] = 0$. 
We assume the noise-norm $\|{\bf N}_i(\bu)\|$ is almost surely bounded, i.e., 
\an{for every $i$, there is a scalar $c_i>0$ such that
every time a noisy subgradient is generated}  we have with probability~1,
\begin{align}\label{eq:noise}
\| \aor{{\bf N}_i({\bf u})\|\le c_i\quad\hbox{for all $\bu\in\R^d$.}}
\end{align} 

We make the assumption  that at each time $t$,  node $i$ can only send messages to its out-neighbors in some
directed graph $G(t)$, where the graph $G(t)$ has vertex set $\{1, \ldots, n\}$  and edge set  $E(t)$. 
We will be assuming that the sequence $\{G(t)\}$ is $B$-strongly-connected, which means
that there is a positive integer $B$ such that the graph with edge set 
\[E_B(k) =  \bigcup_{i=kB}^{(k+1)B-1} E(i) \] 
is strongly connected for each $k \geq 0$. 
Intuitively, we are assuming the time-varying network $G(t)$ must be 
repeatedly connected over sufficiently long time scales. 

We use $N^{\rm in}_i(t)$ and $N^{\rm out}_i(t)$ denote the in- and out-neighborhoods 
of node $i$ at time $t$, respectively, where by convention node $i$ is always considered to be an in- and out- neighbor of itself,
so $i \in N^{\rm in}(i)(t), i \in N^{\rm out}(t)$ for all $i,t$. 
We use $d_i(t)$ to denote the out-degree of node $i$, and we assume
that every node $i$ knows its out-degree $d_i(t)$ at every time $t$. 

We will analyze a version of the subgradient-push method of~\cite{AA2013}, where
each node $i$ maintains
vector variables $\bz_i(t), \bx_i(t), \bw_i(t)\in\R^d$,  as well as a scalar variable 
$y_i(t)$. These
quantities are updated according to the following rules:
for all $t\ge0$ and all $i=1,\ldots,n$, 
\begin{align}\label{eq:minmet} 
\bw_i(t+1) & =  \sum_{j \in N_i^{\rm in}(t)} \frac{\bx_j(t)}{d_j(t)},\cr
y_i(t+1) & =  \sum_{j \in N_i^{\rm in}(t)} \frac{y_j(t)}{d_j(t)}, \cr
\bz_{i}(t+1) & =  \frac{\bw_{i}(t+1)}{y_{i}(t+1)},\cr
\bx_i(t+1) &= \bw_i(t+1) - \alpha(t+1)  \bg_i(t+1),  
\end{align} 
where the variables $y_i(t)$ are initialized as  $y_i(0)=1$ for all $i$. 
Here, we use $\bg_i(t+1)$ to abbreviate the notation $\bg_i(\bz_i(t+1))$ (see Eq.~\eqref{eq:noisy-grad}). 
The positive stepsize $\alpha(t+1)$ will be specified later. 

These updates have a simple physical implementation: each node $j$ broadcasts the 
quantities $\bx_j(t)/d_j(t), y_j(t)/d_j(t)$ to all of the nodes $i$ in its out-neighborhood. Each neighbor 
$i$ then sums the received messages to obtain $\bw_i(t+1)$ and $y_i(t+1)$. 
The updates of $\bz_i(t+1), \bx_i(t+1)$  \aor{then} do not require any \aor{additional} communications 
among the nodes at step~$t$. 

\an{
To provide insights into the method, lets us focus on push-sum method for distributed averaging.
Suppose we have a homogeneous and ergodic Markov chain with a single recurrent class, and let 
$A$ be its transition matrix ($A$ is column-stochastic). 
Also, suppose the nodes of the chain have some initial values $x_i(0)\in\re$. Let $x(0)$ be a vector of these values, 
and consider the following linear dynamic, initiated with $x(0)$:
\[x(t+1)=Ax(t).\]
Since the chain is ergodic, the matrices $A^t$ converge to a rank-one matrix with identical columns, i.e.,
$\lim_{t\to\infty} A^t = \pi \1'$, where $\pi$ is a stochastic vector with $\pi_i>0$ for all $i$.
Therefore, for $x(t)$ we have 
\[\lim_{t\to\infty} x(t)=\left(\lim_{t\to\infty} A^t \right)x(0)=\1'x(0)\,\pi.\]
Suppose now, we replicate the dynamics from a different initial state, say $y(0)\in\re^n$, and we let 
$y(t+1)=Ay(t)$. Then, we have 
\[\lim_{t\to\infty} y(t)=\left(\lim_{t\to\infty} A^t \right)y(0)=\1'y(0)\,\pi.\]
Consider the coordinate-wise ratio of the vectors $x(t)$ and $y(t)$. The limits of these ratios satisfy the following relation
\begin{equation}\label{eq:ratio}\lim_{t\to\infty} \frac{x_i(t)}{y_i(t)}=\frac{\1'x(0)}{\1'y(0)},\end{equation}
which relies on the fact that $\pi_i>0$ (these values cancel out in the ratio). Thus, any influence that the chain may induce, as reflected in different steady state values $\pi_i$, do not appear in the limiting ratios of $x_i(t)/y_i(t)$. 
Furthermore, as indicated by Eq.~\eqref{eq:ratio}, if we set $y(0)=\1$, we obtain
\[\lim_{t\to\infty} \frac{x_i(t)}{y_i(t)}=\frac{\1'x(0)}{n},\]
showing that the ratios $x_i(t)/y_i(t)$ approach the initial average as $t\to\infty$.
When the underlying Markov chain is time-varying (i.e., the matrix $A$ is time-varying, then one would expect
that the limits of the ratios $x_i(t)/y_i(t)$ track the running averages $\frac{\1'x(t)}{n}$ with increasing accuracy.}

\an{The algorithm~\eqref{eq:minmet} is motivated by the insight that 
the ratios $x_i(t)/y_i(t)$ can track the running averages $\frac{\1'x(t)}{n}$ 
with the accuracy that can been characterized by the 
connectivity stricture of the underlying graphs. Additionally, the running averages are controlled by the "gradient" field in order  to move them toward the set of optimal solutions of the problem 
of our interest. Specifically, in the light of the above discussion, the updates $\bx_i(t+1)$ and $y_i(t+1)$ in the algorithm~\eqref{eq:minmet} 
correspond to updates of vector-variables of the nodes, where $y$ variables serve to cancel out the effects of scaling
(which is due to a time-varying Markov chain, as reflected by the graph structure).
The updates of $z_i(t+1)$ in algorithm~\eqref{eq:minmet} are just keeping track of the ratios at the nodes 
(as these will be consenting in a long run). 
The last update step in algorithm~\eqref{eq:minmet}
is basically forcing the consensus point to asymptotically approach an optimal solution of the problem.} 

%For more details on subgradient-push and its motivation we refer the reader to 
%the paper~\cite{AA2013}. \aor{Here we only briefly mention the intuition behind the method. 
%A natural approach to distributed optimization is to combine a subgradient descent at each node 
%with some kind of consensus iteration which brings the nodes closer together. Unfortunately,
%doing this runs the possibility of favoring 
%the subgradients of those nodes at ``central'' positions in the graph which can strongly influence 
%the remaining nodes. One solution is to mix gradient descent with a consensus update using a doubly
%stochastic matrix; this succeeds at ensuring that all nodes are treated equally, but unfortunately
%there is no way to implement this update over time-varying directed graphs. The subgradient-push solves the problem by keeping track of 
%the auxiliary scalars $y_i(t)$, which are then used to rescale the subgradient updates so that,} 
%\an{in a long run, the subgradients are weighted equally at each update.}

Our previous work in~\cite{AA2013} provided a rate estimate for a suitable averaged 
version of the variables ${\bf z}_i(t)$ with the  stepsize choice
$\a(t)=1/\sqrt{t}$. In particular, we showed in [15] that, for each $i=1, \ldots, n$, 
a suitably averaged version of 
${\bf z}_i(t)$ converges to the same global minimum of the function $F(z)$ at a rate of $O((\ln t)/\sqrt{t})$.
Our main contribution  in this paper is an improved convergence rate estimate $O( (\ln t)/t)$ under \an{the strong convexity}
assumption on the functions $f_i$.

Recall that a convex function $f:\R^d\to\R$ is $\mu$-strongly convex  with $\mu > 0$
if the following relation holds for all $\bx,\by\in\R^d$:
\[f(\bx) - f(\by)\ge g'(\by)(\bx-\by)+\frac{\mu}{2}\|\bx-\by\|^2,\]
where $g(\by)$ is any subgradient of $f(\bz)$ at $\bz=\by$. 
%Note that a function
%is not required to be differentiable to be strongly convex.

We next provide precise statements of our improved rate estimates.   
For convenience, we define 
 $$\bar \bx(t)= \frac{1}{n} \sum_{j=1}^n {\bf x}_j(t) $$ 
to be the vector which averages all the $\bx_j(t)$ at each node. 
Furthermore, 
let us introduce some notation for the assumptions we will be making.

\begin{assumption} \label{assume:main1}$~$ \\
(a)
The graph sequence $\{G(t)\}$ is $B$-strongly-connected. \\
(b)
Each function $f_i$ is $\mu_i$-strongly convex with $\mu_i > 0$. 
\end{assumption}

%Assumption 1(a) is clearly necessary in order to obtain the requisite repeated connectivity in the network. 
Note that Assumption~1(b) implies the existence of a unique global minimizer ${\bf z^*}$ of $F({\bf z})$.

One of our technical innovations will be to resort to a somewhat unusual type of averaging
motivated by the work in~\cite{ANSL2013}. 
Specifically, we will require each node to maintain the variable $\widehat \bz_i(t) \in \R^d$ defined by
\begin{equation}\label{eq:delayed-aver}
{\widehat \bz}_i(t)=\frac{\sum_{s=1}^{t} (s-1) \bz_i(s)}{t(t-1)/2}\quad\hbox{for $t\ge 2$}.\end{equation}
This can easily be done recursively, e.g., by  setting $\widehat \bz_i(1)=\bz_i(0)$ and updating as
\begin{equation}\label{eq:aver-seq} \widehat \bz_i(t+1) = \frac{t\bz_i(t+1) + S(t) \widehat \bz_i(t)}{S(t+1)}
\quad\hbox{for $t\ge1$},  \end{equation} where $S(t) = t(t-1)/2$ for $t\ge2$. 

We are now ready to state our first main result, which deals with the speed at which 
the averaged iterates $\widehat \bz_i(t)$ we have just described converge to
the global minimizer ${\bf z}^*$ of $F({\bf z})$.

\begin{theorem} \label{thm2} 
Suppose Assumption 1 is satisfied and 
$\alpha(t) =\frac{p}{t}$ for $t\ge1$, where the constant $p$ is such that 
\begin{equation} \label{eq:step} p\,\frac{\sum_{i=1}^n \mu_i}{n}\ge 4. \end{equation} 
Suppose further that there exists a scalar $D$ such that with probability $1$,
$\sup_t \|{\bf z}_i(t)\| \leq D$.  Then, we have for all $i=1, \ldots, n$,  \aor{
\begin{align*} 
&\mathbb{E} 
\left[ F({\widehat \bz}_i(\tau))-F(\bz^*) +\sum_{j=1}^n \mu_j \|{\widehat \bz}_j(\tau) - \bz^*\|^2 \right] \cr
&\le  \frac{80  L }{\tau \delta}\,\frac{\lambda}{1-\lambda} \sum_{j=1}^n\|\bx_j(0)\|_1 \cr 
& +\frac{80p L n \max_j B_j  }{\tau \delta (1 - \lambda) }  (1+\ln (\tau-1))   
+ \frac{p }{\tau }\sum_{j=1}^n \an{(L_j +c_j)^2},\end{align*} }
\an{where $L_i$ is the largest-possible Euclidean norm of any subgradient of $f_i$ on the ball of radius $D$ around 
the origin, $L=\sum_{j=1}^n L_j$, $B_i=\sqrt{d}(L_i+c_i) $}, 
% $B_i = \aor{\sqrt{d}( L_i +  c_i)}$, not sure why we need the sort{d}??
while the scalars $\lambda \in (0,1)$ and $\delta > 0$
are functions of the graph sequence $\{G(t)\}$ 
which satisfy
\begin{align*} \delta \geq  \frac{1}{n^{nB}}, \qquad
\lambda \leq \left( 1 - \frac{1}{n^{nB}} \right)^{1/(nB)}.
\end{align*}
Moreover, if each of the graphs $G(t)$ is regular\footnote{A directed graph $G(t)$ is regular if  
every out-degree and every in-degree of a node in $G(t)$ equals $d(t)$ for some $d(t)$.}, then 
\begin{align*} 
\delta  =  1, \  
\lambda \leq 
\min\left\{\left( 1 - \frac{1}{4n^3} \right)^{1/B} , \max_{t\ge0} \sigma_2(A(t)) \right\},
\end{align*} where $A(t)$ is defined by 
\begin{equation}\label{eq:defA}
A_{ij}(t) =\left\{
\begin{array}{ll}
1/d_j(t) & \hbox{whenever $j \in N_i^{\rm in}(t)$},\cr 
0 &\hbox{otherwise}.\end{array}\right.
\end{equation} 
and $\sigma_2(A)$ is the second-largest singular value of  $A$.
\end{theorem}

Note that each term on the right-hand side of the bound in the above theorem has a $\tau$ in the
denominator and \an{an} $\ln (\tau-1)$ or a constant in the numerator.  The convergence time above should therefore be interpreted as proving a decay with time which 
decreases at an expected $O( (\ln t)/t)$ rate with the number of iterations~$t$. Note that this is an extended and corrected version of a result from
the conference version of this paper~\cite{our_conf}.

\aor{We remark that our result is new even for undirected graphs, which are included as a special case of the 
above theorem; indeed, to our knowledge, we are the first to demonstrate a decay rate of $O( (\ln t)/t)$ for stochastic gradient
descent over time-varying undirected graphs (however, recall our earlier discussion of \cite{Tsianos2012, ermin} which derived similar decay rates for the deterministic case with a fixed undirected graph). Surprisingly, the undirected and directed cases do not appear to be very different; the main difference seems to do with the constant $\delta$ and $\lambda$. Indeed, note that} the bounds for the constants
$\delta$ and $\lambda$  which appear in the bound are rather large in the most  general case; in particular, they grow exponentially in
the number of nodes $n$. At present, this scaling appears unavoidable: those constants reflect the best available bounds
on the performance of average consensus protocols in directed graphs, and  it is an open question whether 
average consensus on directed graphs can be done in time polynomial in $n$. In the case of regular graphs, the bounds scale polynomially in $n$ due to the availability of good bounds on the convergence of consensus.  \aor{Similarly, for undirected case, the possibility of modifyin the protocol by instead choosing a  symmetric matrix $A$ leads to good bounds on $\delta$ and $\lambda$ \cite{noot09}}. Our results therefore further motivate problem of finding consensus algorithms with good convergence times, \aor{especially on directed graphs}.

Finally, we remark that 
choosing a stepsize parameter $p$ so that Eq.~(\ref{eq:step}) is satisfied is most easily done by instead insuring that 
$p (\min_i \mu_i) /n > 4$. This is a more conservative condition than 
that of Eq.~(\ref{eq:step}) but ensuring it requires the nodes only to compute $\min_i \mu_i$. This is more convenient because the minimum of any collection of numbers $r_1, \ldots, r_n$ (with $r_i$ stored at node $i$) can be easily computed by the following distributed protocol: node $i$ sets its initial value to $r_i$ and then
repeatedly replaces its value with the minimum of the values of its in-neighbors. It is easy to see that on any fixed network, this process 
converges to the minimum in as many steps as the diameter. Furthermore, on any $B$-strongly-connected sequence this processes 
converges in the optimal $O(nB)$ steps. Thus, the pre-processing required to come up with a suitable step-size parameter $p$ is reasonably small. 

A shortcoming of Theorem 1 is that we must {\em assume} that the iterates ${\bf z}_i(t)$ remain bounded (as opposed to obtaining this as a by-product of the theorem).  This is a common situation in the analysis of subgradient-type methods in non-differentiable optimization: the boundedness of the iterates or their subgradients often needs to be assumed in advance in order to obtain a result about convergence rate. 

We next show that we can remedy this shortcoming at the cost of imposing additional assumptions on the functions $f_i$, namely that they are differentiable and their gradients are Lipschitz.

\begin{assumption} \label{assume:main2} Each $f_i$ is differentiable and its gradients
are  Lipschitz continuous, i.e., for a scalar $M_i>0$,
\[\|\nabla f_i(\bx)-\nabla f_i(y)\|\le M_i \|x-y\|\qquad\hbox{for all $x,y\in\R^d$}.\]
\end{assumption}

\begin{theorem} \label{thm0} 
Suppose that Assumption 1 and Assumption 2 hold and suppose $\lim_{t\to \infty}\alpha(t)=0$. 
Then, there exists a scalar $D$ such that
with probability $1$, $\sup_{t} \|\bz_i(t)\| \leq D$ for all $i$.
\end{theorem}

The proof of this theorem is constructive in the sense than an explicit expression for $D$ can be derived in terms of the 
\an{level set} growth of the functions $f_j$. \aor{Additionally, the scalar $D$ depends on the initial points $\bx_i(0)$, the step-size sequence $\alpha(t)$,
the functions $f_i(\cdot)$,} \an {the Lipschitz constants $M_j$ and the noise bounds $c_j$ (cf.~\eqref{eq:noise}). }

 Putting Theorems 1 and 2 together, we obtain our main result: for strongly convex functions with Lipschitz gradients, the \aor{stochastic} (sub)gradient-push with appropriately chosen step-size and averaging strategy
converges at an $O((\ln t)/t)$ rate.

%%%%%%%%%%%%%%%%%%%%%%%%%%%%%%%%%
\section{Proof of Theorem \ref{thm2}}\label{sec:proof1}
%%%%%%%%%%%%%%%%%%%%%%%%%%%%%%%%%

We briefly sketch the main ideas of the proof of Theorem \ref{thm2}. First, we will argue that if the subgradient terms in the subgradient-push protocol are bounded, then as a consequence of the decaying stepsize $\alpha(t)$, the protocol will achieve consensus. We will then analyze the evolution of the average $\bar \bx(t)$ and show that, as a consequence of the protocol achieving consensus,  $\bar \bx(t)$ satisfies approximately the same recursion as the iterates of the ordinary subgradient method. Finally, \aor{the key idea in our proof is the observation in}~\cite{ANSL2013}  that, for a noisy gradient update on a strongly convex function, a decay of $O(1/t)$ 
can be achieved by a simple averaging of iterates that places more weight on recent iterations, 
specifically \an{by weighting the $t$'th iterate proportional to $t$}. \aor{Here, we show that, for the perturbed subgradient method which is followed by} 
\an{the averaging step $\bar \bx(t)$, a nearly identical rate $O((\ln t)/t)$ can be achieved.}
 
Our starting point is an analysis of a perturbation of the so-called push-sum protocol of~\cite{dobra-kempe} 
for computing averages  in  directed networks. We next describe this perturbed 
push-sum protocol. Every node $i$ maintains {\em scalar }
variables $x_i(t), y_i(t), z_i(t), w_i(t)$, where $y_i(0)=1$ for all $i$. 
Every node $i$ updates these variables according to the following rule: for $t\ge0$,
\begin{eqnarray}\label{eq:newx}
 w_i(t+1)  &=& \sum_{j \in N^{\rm in}_i(t) } \frac{x_j(t)}{d_j(t)},\cr
 \hbox{}\cr
y_i(t+1) & = & \sum_{j \in N^{\rm in} _i(t)} \frac{y_j(t)}{d_j(t)}, \cr
\hbox{}\cr
z_i(t+1) & = & \frac{w_i(t+1)}{y_i(t+1)},  \cr
\hbox{}\cr
x_i(t+1)&=&w_i(t+1) + \e_i(t+1),
\end{eqnarray} 
where $\e_i(t)$ is some (perhaps adversarially chosen) perturbation at time $t$. 
Without the perturbation term $\e_i(t)$, the method in Eq.~(\ref{eq:newx}) is called {\em push-sum}.  
For the perturbed push-sum method above in Eq.~(\ref{eq:newx}), we have that the following is true.

\begin{lemma}[\cite{AA2013}]\label{lemma:newx} 
Consider the sequences $\{z_i(t)\}$, $i=1,\ldots,n,$ generated by the method in 
Eq.~\eqref{eq:newx}. Assuming that the graph sequence $\{G(t)\}$ is
$B$-strongly-connected, we have that
for all $t\ge 1$,
\begin{eqnarray*} 
\left|z_i(t+1) -\frac{\1' x(t)}{n}\right| 
& \le &
\frac{8}{\delta}\,\left(\lambda^t\|x(0)\|_1 \right.
\cr
&& \left.
+ \sum_{s=1}^{t}\lambda^{t-s}\|\e(s)\|_1 \right),
\end{eqnarray*} 
where $\e(s)$ is a vector in $\R^n$ which stacks up the scalar variables $\epsilon_i(s)$, $i=1, \ldots, n$, and $\delta, \lambda$ satisfy the same inequalities as in Theorem \ref{thm2}.
\end{lemma}

We refer the reader to ~\cite{AA2013} for a proof where this statement is Lemma 1. \aor{Informally, 
the push-sum protocol} 
\an{ensures that all $z_i(t)$ track the running averages $\frac{\1' x(t)}{n}$ with a geometric rate $\lambda$, 
while the perturbations $\epsilon_i(t)$ push 
the node values apart. The perturbations can be viewed as an external force that influences the node values and causes additional disagreement.
Lemma~\ref{lemma:newx} provides a bound on the size of the disagreements among the agents in terms of the network caused imbalances and the imbalances due to the external force.} 

\begin{corollary} \label{cor:vector} 
Consider the update of Eq.~\eqref{eq:newx} with the scalar variables $x_i(t), w_i(t), z_i(t), \e_i(t)$ 
replaced by the vector variables
$\bx_i(t), \bw_i(t), \bz_i(t), {\bf e}_i(t)$ for each $i=1, \ldots, n$. 
Assuming that the graph sequence $\{G(t)\}$ is
$B$-strongly-connected,
for all $i = 1, \ldots, n, t\ge 1$ we have 
\begin{align*}
\left\|\bz_i(t+1) -\frac{\sum_{j=1}^n\bx_j(t)}{n}\right\|
&\le
\frac{8}{\delta}\,\left(\lambda^t\sum_{\aor{j}=1}^n \|\bx_\aor{j}(0)\|_1 \right.
\cr
& \left.
+ \sum_{s=1}^{t}\lambda^{t-s}\sum_{\aor{j}=1}^n\|\bold{e}_\aor{j}(s)\|_1 \right),
\end{align*}  
where $\delta, \lambda$ satisfy the same inequalities as in Theorem~\ref{thm2}.  
\end{corollary}

Corollary~\ref{cor:vector} follows \aor{immediately} by applying Lemma~\ref{lemma:newx} to each coordinate 
of $\R^d$ and by using the fact that the Euclidean norm of any vector is at most as large as the $1$-norm. 
A more specific setting when the perturbations ${\bf e}_i(t)$ decay with $t$ is considered in the following corollary.  

\begin{corollary} \label{cor:error} 
Under the assumptions of Corollary~\ref{cor:vector} and assuming that 
the perturbation vectors $\bold{e}_i(t)$ are  vectors satisfying
for some scalar $D>0$,
\[ \mathbb{E} \left[ \|\bold{e}_i(t)\|_1 \right]\leq \frac{D}{t} \ 
\qquad\hbox{for all $i=1,\ldots,n$ and all $t\ge1$,}\]
we \aor{then} have that for all $i=1,\ldots,n$ and all $\tau\ge1,$ 

\begin{align*} 
&\mathbb{E} \left[
\sum_{t=1}^{\tau} \left\|\bz_i(t+1) -\frac{\sum_{j=1}^n\bx_j(t)}{n}\right\| \right] \cr
&\qquad\le \frac{8}{\delta}\,\frac{\lambda}{1-\lambda} \sum_{j=1}^n\|\bx_j(0)\|_1 
+\frac{8}{\delta}\,\frac{Dn}{1-\lambda}(1+\ln \tau).
\end{align*} 
The parameters $\delta>0$ and $\lambda\in (0,1)$ satisfy the same inequalities as in Theorem~\ref{thm2}.
\end{corollary}

\begin{proof} 
We use Corollary~\ref{cor:vector}. 
The first term in the estimate follows immediately. The second term requires some attention:  
\begin{eqnarray*} 
\mathbb{E} \left[ \sum_{t=1}^{\tau} \sum_{s=1}^{t}\lambda^{t-s} 
\sum_{j=1}^n\|\bold{e}_j(s)\|_1 \right]
& \le &  Dn\sum_{t=1}^{\tau} \sum_{s=1}^{t} \frac{\lambda^{t-s}}{s}\\ 
%& = & Dn\sum_{s=1}^{\tau} \left(\sum_{k=0}^{\tau-s}\lambda^k \right) \frac{1}{s} \\ 
& \le &  \frac{Dn}{1-\lambda}\sum_{s=1}^{\tau} \frac{1}{s} 
\end{eqnarray*} and the result follows from the usual bound on the sum of harmonic series, 
$\sum_{s=1}^{\tau} \frac{1}{s} \le 1+\ln \tau$. 
\end{proof} 

%%%%%%%%%%%%%%%%%%%%%%%%%%%%%%%%%%%%%%%%%%%%%%
%\section{Results for Subgradient-Push Method}\label{sec:optim}
%%%%%%%%%%%%%%%%%%%%%%%%%%%%%%%%%%%%%%%%%%%%%%
In the proof of Theorem~\ref{thm2}, we also use 
the following result, which is a generalization of Lemma~8 in~\cite{AA2013}. Before stating this lemma, we introduce some notation. 
%We will adopt the shorthand $\widehat \bg_j(k)$ for $\widehat \bg_j(\bz_j(k))$ and similarly $\bg_j(k)$ for $\bg_j(\bz_j(k))$. Finally, 
We define  ${\cal F}_t$ to be all the information generated by
the stochastic gradient-push method by time $t$, i.e., all the $\bx_i(k), \bz_i(k), \bw_i(k), y_i(k),  \bg_i(k)$ 
and so forth for $k=1, \ldots, t$.
We then have the following lemma. 

\begin{lemma}\label{lemma:key}
\an{Assume that there is a scalar $D>0$ such that 
$\sup_{t}\|\bz_i(t)\|\le D$ for all $i$ with probability 1.
Then, we have  $\sup_t\|\bar \bx(t)\|\le D$ with probability 1. Furthermore,}
\aor{if Assumption~\ref{assume:main1}(b) holds,}
then \an{we have with probability 1,} for all $\bv\in\mathbb{R}^d$ and $t\ge0$, 
\begin{align*}
\mathbb{E} \left[ \right. \|\bar\bx(t+1) & - \bv \|^2 ~|~ {\cal F}_t \left. \right] 
\le   \|\bar\bx(t) - \bv \|^2 \cr & - \frac{2\alpha(t+1)}{n} \left(F(\bar\bx(t)) - F(\bv) \right)\cr
& 
- \frac{\alpha(t+1)}{n} \sum_{j=1}^n \mu_j \|\bz_j(t+1) - \bv\|^2 \cr
& + \frac{4\alpha(t+1)}{n} \sum_{j=1}^n L_j \|\bz_j(t+1)-\bar\bx(t)\|  \cr &
 +\frac{\alpha^2(t+1)}{n}\sum_{j=1}^n \an{(L_j + c_j)^2},
\end{align*}  
where \an{$L_j=\max_{\|u\|\le D}\nabla f_j(u)$ and constants $c_j$ are  from~\eqref{eq:noise}.}
%from \aor{Theorem~\ref{thm2}}.
\end{lemma}

\begin{proof} 
\an{Note that the matrix $A(t)$ defined in the statement of Theorem~1 (see Eq.~\eqref{eq:defA}) 
is column stochastic, so that $\1'u=\1'A(t)u$ for any vector $u \in \R^n$.
We next show that when $\{\bar\bz_j(t)\}$ are bounded for all $j$ with probability 1, so are the averages $\bar x(t)$. To see this,
 we note that by the definition of $\bw_i(t+1)$  and the column-stochasticity of $A(t)$, we have
 $\sum_{i=1}^n \bw_i(t+1)=\sum_{j=1}^n x_j(t),$ implying that
 \begin{align}\label{eq:oneoo}
 \bar \bx(t)= \frac{1}{n}\sum_{i=1}^n \bw_i(t+1)  = 
  \frac{1}{n}\sum_{i=1}^n y_i(t+1) \bz_i(t+1) ,\end{align}
 where the last equality follows from the definition of $\bz_i(t+1)$. Since the matrices $A(t)$ are column stochastic, the sums of $y_i(t)$ are preserved at all times, i.e., $\sum_{i=1}^n y_i(t)=n$ for all $t$. Furthermore, $y_i(t) >0$ for all $i$ and $t$. Thus,
 relation~\eqref{eq:oneoo} shows that each vector $\bar \bx(t)$ is a convex combination of $\bz_i(t+1)$, $i=1,\ldots,n,$
 implying that for all $t\ge0$
 \begin{align}\label{eq:xbound}
 \|\bar \bx(t)\|\le \max_{i}\|\bz_i(t+1)\|.\end{align}
 Thus, with probability 1, we have $\|\bar x(t)\|\le D.$ }

\an{
We next show the relation stated in the lemma. Due to the column-stochasticity of the matrices $A(t)$,
 for the stochastic gradient-push update of Eq.~(\ref{eq:minmet}) we have}
\begin{equation} \label{eq:avdone} 
\bar \bx(t+1) = \bar \bx(t) - \frac{\alpha(t+1)}{n} \sum_{j=1}^n\bg_j(t+1). 
\end{equation}

Now, let $\bv\in\mathbb{R}^d$ be an arbitrary vector. 
From relation~\eqref{eq:avdone} we can see that for all $t\ge0$,
\begin{align*}
&\|\bar\bx(t+1)  - \bv\|^2 \le \|\bar\bx(t) - \bv \|^2 \cr 
&\qquad - \frac{2\alpha(t+1)}{n} \sum_{j=1}^n \bg_j(t+1)' (\bar\bx(t) - \bv)  \cr 
&\qquad   +\frac{\alpha^2(t+1)}{n^2} \|\sum_{j=1}^n \bg_j(t+1)\|^2.
\end{align*} 
Taking expectations of both sides with respect to ${\cal F}_t$, 
and using $\bg_j(t+1) =\nabla f_j(\bz_j(t+1)) + {\bf N}_j(\bz_j(t+1))$ (see Eq.~\eqref{eq:noisy-grad}) and
the relation 
\[\mathbb{E}[{\bf N}_j({\bz_j(t+1))}\mid {\cal F}_t] 
%= \mathbb{E}[{\bf N}_j({\bz_j(t+1))}\mid z_j(t+1)]
=0,\]
we obtain 
\begin{align*}
&\mathbb{E} \left[ \right.
\|\bar\bx(t+1)   - \bv\|^2 \mid {\cal F}_t \left. \right] 
\le  \|\bar\bx(t) - \bv \|^2  \cr 
& \qquad - \frac{2\alpha(t+1)}{n} \sum_{j=1}^n  
\nabla f_j(\bz_j(t+1))' (\bar\bx(t) - \bv)  \cr 
&  \qquad   +\frac{\alpha^2(t+1)}{n^2} 
\aor{\mathbb{E} \left[ \| \sum_{j=1}^n \bg_j(t+1)\|^2 ~|~ {\cal F}_t \right] }.
\end{align*} 
Next, we upper-bound the last term in the preceding relation. By using 
the inequality $ (\sum_{j=1}^n a_j)^2 \leq n \sum_{j=1}^n a_j^2$ 
we obtain \aor{that with probability $1$},
\[\| \sum_{j=1}^n \bg_j(t+1)\|^2\le n \sum_{j=1}^n \|\bg_j(t+1)\|^2\le \aor{n} \sum_{j=1}^n \an{(L_j+c_j)^2},\]
since \an{$\|\bg_j(t+1)\| \leq \max_{\|u\|\le D}\|\nabla f_j(u)\|+c_j$} 
from Eq. (\ref{eq:noise}). % and the definition of $B_j$. 
This implies for all $t\ge1$,
\begin{align}\label{eq:beg1}
&\mathbb{E} \left[ \right.
\|\bar\bx(t+1)   - \bv\|^2 \mid {\cal F}_t \left. \right] 
\le  \|\bar\bx(t) - \bv \|^2  \cr 
& \qquad - \frac{2\alpha(t+1)}{n} \sum_{j=1}^n  
\nabla f_j(\bz_j(t+1))' (\bar\bx(t) - \bv)  \cr 
&  \qquad +\frac{\alpha^2(t+1)}{n} \sum_{j=1}^n \an{(L_j+c_j)^2}.
\end{align} 
Now, consider each of the cross-terms $\nabla f_j(\bz_j(t+1))' (\bar\bx(t) - \bv)$ in~\eqref{eq:beg1},  
for which  we write
\begin{align}\label{eq:beg2}
& \nabla f_j(\bz_j(t+1))' (\bar\bx(t) - \bv) \cr
& =\nabla f_j(\bz_j(t+1))' (\bar\bx(t) - \bz_j(t+1)) \cr 
&  \quad + \nabla f_j(\bz_j(t+1))' (\bz_j(t+1) - \bv).
\end{align} 
%Recall that the vector $\bg_j'(t+1)$ is a subgradient of $f_i$ at $\bz_j(t+1)$. 
%Note that as a consequence of Eq. (\ref{eq:minmet}), we have that \aor{$y_j(t), j = 1, \ldots, n$ are 
%nonnegative real numbers which sum to $n$ for all $t$. Therefore,
%$$\bar \bx(t)= \frac{1}{n}\sum_{j=1}^n \bw_j(t+1)  = \sum_{j=1}^n \frac{y_j(t+1)}{n} \bz_j(t+1) ,$$
%so that 
By using the Cauchy-Schwarz inequality we have with probability 1,
\begin{eqnarray}\label{eq:beg3}
&&\nabla f_j(\bz_j(t+1))' (\bar\bx(t) - \bz_j(t+1))\qquad\cr
%& \ge &-\|\nabla f_j(\bz_j(t+1))\| \|\bar\bx(t) - \bz_j(t+1)\|\cr
%& \ge & -\max_{\|u\|\le D} \|\nabla f_j(u)\|\|\bar\bx(t) - \bz_j(t+1)\|\cr
& \ge & - L_j \|\bar\bx(t) - \bz_j(t+1)\|,
\end{eqnarray}
As for the term $\nabla f_j(\bz_j(t+1))' (\bz_j(t+1) - \bv)$, we
use the fact that the function $f_j$ is $\mu_i$-strongly convex  
to obtain 
\begin{align}\label{eq:a1}
\nabla f_j(\bz_j(t+1))' (\bz_j(t+1) - \bv) 
&\ge  f_j(\bz_j(t+1)) - f_j(\bv) \cr
&\ +\frac{\mu_j}{2}\|\bz_j(t+1) - \bv\|^2.\end{align}
\an{By writing $ f_j(\bz_j(t+1)) - f_j(\bv) = (f_j(\bz_j(t+1)) - f_j(\bar\bx(t))  + (f_j(\bar\bx(t)) - f_j(\bv))$
and by using the convexity of $f_j$, we have
\begin{align}\label{eq:a2}
f_j(\bz_j(t+1)) - f_j(\bv) 
&\ge \nabla f_j(\bar\bx(t))'(\bz_j(t+1) - \bar\bx(t) ) \cr
 & + (f_j(\bar\bx(t)) - f_j(\bv)),
 \end{align}
 where $\nabla f_j(\bar\bx(t))$ is a subgradient of $f_j$ at $\bar\bx(t)$.
 In view of relation~\eqref{eq:xbound}, the sub-gradients $\nabla f_j(\bar x(t))$ are also bounded with probability 1, so we have
 \begin{align}\label{eq:a3}
 \nabla f_j(\bar\bx(t))'(\bz_j(t+1) - \bar\bx(t) )  - L_j\|\bz_j(t+1)-\bar\bx(t)\|.\end{align}
 From relation~\eqref{eq:a1}, \eqref{eq:a2} and \eqref{eq:a3}, we conclude that with probability 1,}
\begin{align} \nonumber
\nabla f_j(\bz_j(t+1))'
&  (\bz_j(t+1)-  \bv) \geq   - L_j\|\bz_j(t+1)-\bar\bx(t)\|   \cr & + f_j(\bar\bx(t)) - f_i(\bv) \\ & + \frac{\mu_j}{2}\|\bz_j(t+1) - \bv\|^2. \label{eq:beg4}
\end{align} 
By substituting the estimates of Eqs.~\eqref{eq:beg3} and~\eqref{eq:beg4} 
back in relation~\eqref{eq:beg2}, and using $F(\bx)= \sum_{j=1}^n f_j(\bx)$
we obtain 
\begin{align*} 
&\sum_{i=1}^n \nabla f_j(\bz_j(t+1))' (\bar\bx(t)  -  \bv)
\ge  F(\bar\bx(t)) - F(\bv) \cr 
&  \quad + \frac{1}{2}\sum_{j=1}^n \mu_j \|\bz_j(t+1) - \bv\|^2\cr
&\quad -  2\sum_{j=1}^n L_j \|\bz_j(t+1)-\bar\bx(t)\|.  
\end{align*} 
Plugging this relation into Eq.~\eqref{eq:beg1}, we obtain the statement of this lemma. 
\end{proof}

With Lemma~\ref{lemma:key} in place,
we are now ready to provide the proof of  Theorem~\ref{thm2}.
Besides Lemma~\ref{lemma:key}, our arguments will also crucially rely on the results 
established earlier for the perturbed push-sum method. 

\begin{proof}[Proof of Theorem \ref{thm2}] 
The function $F=\sum_{i=1}^n f_i$ has a unique minimum which we will 
denote by $\bz^*$. In Lemma~\ref{lemma:key} we let $\bv=\bz^*$ to obtain for all $t\ge0$,
\begin{align}\label{eq:main2}
\mathbb{E} \left[ \right. \|\bar\bx(t+1) & - \bz^* \|^2 ~|~ {\cal F}_t \left. \right]
\le   \|\bar\bx(t) - \bz^* \|^2  \cr &  - \frac{2\alpha(t+1)}{n} \left(F(\bar\bx(t)) - F(\bz^*) \right) 
\cr & - \frac{\alpha(t+1)}{n} \sum_{j=1}^n \mu_j \|\bz_j(t+1) - \bz^*\|^2\cr
& + \frac{4\alpha(t+1)}{n} \sum_{j=1}^n L_j \|\bz_j(t+1)-\bar\bx(t)\| \cr & 
 +\an{\frac{\alpha^2(t+1)}{n} \sum_{j=1}^n (L_j+c_j)^2}.\end{align} 
Next, we estimate the term $F(\bar\bx(t)) - F(\bz^*)$ in the above equation
by  breaking it into two parts. On the one hand,  \[F(\bar\bx(t)) - F(\bz^*)\ge \frac{1}{2}\left(\sum_{j=1}^n\mu_j\right)\|\bar\bx(t) - \bz^*\|^2.\]
On the other hand, since the function $F$ is Lipschitz continuous with constant $L = L_1 + \cdots + L_n$ \aor{over the ball of radius $D$ around the origin to which all $\bz_j(t), \overline{\bx}(t)$ always belong} \an{(by assumption and by Lemma~\ref{lemma:key}),}
 we also have that for any $i=1,\ldots,n,$ 
 \begin{align*}
 F(\bar\bx(t)) - F(\bz^*)
 &=\left(F(\bar\bx(t)) - F(\bz_i(t+1)\right)) \cr & +\left(F(\bz_i(t+1))-F(\bz^*)\right)\cr
 &\ge -L\|\bz_i(t+1)-\bar \bx(t)\| \cr & +F(\bz_i(t+1))-F(\bz^*).\end{align*}
 Therefore, using the preceeding two estimates we obtain for all $i=1,\ldots,n$, 
 \begin{align}\label{eq:Fvalues}
 2\left(F(\bar\bx(t)) - F(\bz^*) \right)
 &\ge \frac{1}{2}\left(\sum_{j=1}^n\mu_j\right)\|\bar\bx(t) - \bz^*\|^2\cr
 &-L\|\bz_i(t+1)-\bar \bx(t)\| \cr & +\left(F(\bz_i(t+1))-F(\bz^*)\right).\end{align} 
Combining relation~\eqref{eq:Fvalues} with Eq.~\eqref{eq:main2}, we obtain \aor{ that for each $i=1, \ldots, n$, 
\an{with probability 1,}
\begin{align*}
\mathbb{E} \left[ \right. \|\bar\bx(t+1) & - \bz^* \|^2 ~|~ {\cal F}_t \left. \right]
\le   \|\bar\bx(t) - \bz^* \|^2  \cr &  - \frac{\alpha(t+1)}{n} \frac{1}{2} \left( \sum_{j=1}^n \mu_j \right) ||\overline{\bx}(t) - \bz^*||^2
\cr & - \frac{\alpha(t+1)}{n} \left( F(\bz_i(t+1)) - F(\bz^*) \right)
\cr & + \frac{\alpha(t+1)}{n} L ||\bz_i(t+1) - \overline{\bx}(t) || 
\cr & - \frac{\alpha(t+1)}{n} \sum_{j=1}^n \mu_j \|\bz_j(t+1) - \bz^*\|^2\cr
& + \frac{4\alpha(t+1)}{n} \sum_{j=1}^n L_j \|\bz_j(t+1)-\bar\bx(t)\| \cr & 
 +\frac{\alpha^2(t+1)}{n}\sum_{j=1}^n \an{(L_j+c_j)^2}.\end{align*} 
Now plugging in the expression for $\alpha(t)$ and using the definition of $p$ to combine the first two
terms},  we see that 
for all $i=1,\ldots,n$ and all $t\ge0$, 
 \begin{align*}
\mathbb{E} \left[ \right. &\|\bar\bx(t+1)  - \bz^* \|^2 ~|~ {\cal F}_t \left. \right] \le  {\left(1-\frac{2}{t+1}\right)}
\|\bar\bx(t) - \bz^* \|^2 \cr
&- \frac{p}{n(t+1)}\left(F(\bz_i(t+1))-F(\bz^*)\right) \cr 
&  + \frac{p L}{n(t+1)} \|\bz_i(t+1) - \bar \bx(t) \| \cr & 
- \frac{p}{n(t+1)} \sum_{j=1}^n \mu_j \|\bz_j(t+1) -\bz^*\|^2 \cr
& + \frac{4p}{n(t+1)} \sum_{j=1}^n L_j \|\bz_j(t+1)-\bar\bx(t)\| 
+\frac{p^2}{(t+1)^2}\frac{q^2}{n},
 %& +\frac{p^2}{(t+1)^2}\frac{\sum_{j=1}^n B_j^2}{n},
 \end{align*} 
where \an{$q^2=\sum_{j=1}^n (L_j +c_j)^2$}.
We multiply the preceding relation by $t(t+1)$, and 
%\an{by using
%$\sum_{j=1}^n L_j \|\bz_j(t+1)-\bar\bx(t)\|\le L \sum_{j=1}^n \|\bz_j(t+1)-\bar\bx(t)\|$, where $L=\sum_{j=1}^n L_j$,
we obtain that for all $i=1,\ldots,n$ and all $t\ge1$, 
\begin{align}\label{eq:main4}
&(t+1)t  \mathbb{E} \left[ \right.  \|\bar\bx(t+1)  - \bz^* \|^2 ~|~ {\cal F}_t \left. \right] \cr
& \le t(t-1)\|\bar\bx(t) - \bz^* \|^2 \cr
&- \frac{pt}{n}\left(F(\bz_i(t+1))-F(\bz^*)\right)  + \frac{pLt}{n} \|\bz_i(t+1) - \bar \bx(t) \| \cr 
& - \frac{pt}{n} \sum_{j=1}^n \mu_j \|\bz_j(t+1) - \bz^*\|^2\cr
& + \frac{4pt}{n} \sum_{j=1}^n L_j\|\bz_j(t+1)-\bar\bx(t)\| +\frac{p^2t}{(t+1)}\frac{q^2}{n}.
%& +\frac{p^2t}{(t+1)}\frac{\sum_{j=1}^n B_j^2}{n}.
\end{align} 
\an{By iterating the expectations in~\eqref{eq:main4} and applying the resulting inequality}, recursively, 
we obtain that all $\tau\ge 2$, %\begin{footnotesize}
\begin{align}\label{eq:main5}
\tau &  (\tau-1) \mathbb{E} \|   \left[ \right. \bar\bx(\tau)  - \bz^* \|^2 \left. \right]
\le   \cr
&- \frac{p}{n}\sum_{t=1}^{\tau-1} t \mathbb{E} \left[ F(\bz_i(t+1))-F(\bz^*) 
+\sum_{j=1}^n \mu_j \|\bz_j(t+1) -\bz^*\|^2 \right] \cr
& +  \frac{pL}{n} \sum_{t=1}^{\tau - 1} t \mathbb{E}\left[ \|\bz_i(t+1) - \bar \bx(t)\| \right] \\ & 
 +\frac{4p}{n} \sum_{t=1}^{\tau-1} t \sum_{j=1}^n L_j \mathbb{E} \left[ \|\bz_j(t+1)-\bar\bx(t)\| \right]
 + \frac{p^2 \aor{q^2}}{n}\sum_{t=1}^{\tau-1}\frac{t}{t+1} \notag.\end{align} %\end{footnotesize}
By viewing the stochastic gradient-push method as an instance of the perturbed push-sum protocol, 
we can apply Corollary~\ref{cor:error} with $\bold{e}_i(t)=\a(t)\bg_i(t)$.
Since \an{$\|\bg_i(t)\|_1 \le \sqrt{d}\|\bg_i(t)\|\le B_i$ for all $i, t$, with $B_i=\sqrt{d}(L_i +c_i)$,} 
we see that \aor{with probability~$1$,}
\begin{align*}
\mathbb{E} \left[ \|\bold{e}_i(t)\|_1\right] \le \frac{pB_i}{\aor{t}} \qquad\hbox{for all $i$ and $t\ge1$. }
%= E \left[ ||\bg_i(t) + {\bf N}_i(t)||_1 \right] &  \leq L_i + E \left[||{\bf N}_i(t)||_1 \right] 
%\cr & \leq L_i + \sqrt{d} \sigma_i 
\end{align*} 
Thus, by Corollary~\ref{cor:error} we obtain for all $i=1,\ldots,n$, 
\begin{align*}%\label{eq:last1}
\mathbb{E} \left[ \right.  \sum_{t=1}^{\tau-1} \left\| \right. \bz_i(t+1)  - 
& \frac{\sum_{j=1}^n\bx_j(t)}{n} \left. \right\| \left. \right]
\le \frac{8}{\delta}\,\frac{\lambda}{1-\lambda} \sum_{j=1}^n\|\bx_j(0)\|_1 \cr
&+\frac{8}{\delta}\,\frac{p n ~\aor{\max_j B_j}}{1-\lambda}(1+\ln (\tau-1)).
%&+\frac{8}{\delta}\,\frac{p n ( L_i + \sqrt{d} \sigma_i )}{1-\lambda}(1+\ln (\tau-1)).
\end{align*}
\an{Upon substituting %\eqref{eq:last1} 
the preceding inequality into relation~\eqref{eq:main5} and dividing both sides
by $\tau(\tau-1)$, after re-arranging the terms, we obtain  for all $\tau\ge2,$} 
\aor{
\begin{align*}
& \frac{p}{ n \tau (\tau-1)}   \sum_{t=1}^{\tau-1} t 
\mathbb{E} \left[ \right. F(\bz_i(t+1))- F(\bz^*) 
+\sum_{j=1}^n \mu_j \|\bz_j(t+1) -\bz^*\|^2 \left. \right]   \cr & 
\leq  
 \frac{pL}{n \tau} \frac{8 }{\delta} \left( \frac{\lambda}{1-\lambda} 
 \sum_{j=1}^n \|\bx_j(0)\|_1  + \frac{pn \max_j B_j}{1-\lambda} (1 + \ln (\tau-1) \right) 
\cr & 
+\frac{4p}{n \tau}  \sum_{j=1}^n L_j \frac{8}{\delta} \left( \frac{\lambda}{1-\lambda} 
\sum_{j=1}^n \|\bx_j(0)\|_1  + \frac{pn \max_j B_j}{1-\lambda} (1 + \ln (\tau-1) \right) 
\cr & + \frac{p^2 q^2 }{\tau n}.\end{align*}  
}
\an{
Combining the first two terms on the right hand side of the preceding relation, using $L=\sum_{j=1}^n L_j$ 
and canceling $p/n$ from both sides, we get}
\begin{align} \label{eq:main7} 
&\frac{1}{ \tau (\tau-1)}   \sum_{t=1}^{\tau-1} t 
\mathbb{E} \left[ \right. F(\bz_i(t+1))- F(\bz^*) 
+\sum_{j=1}^n \mu_j \|\bz_j(t+1) -\bz^*\|^2 \left. \right]   \cr & 
\leq  \frac{5L}{\tau}  \frac{8}{\delta} \left( \frac{\lambda}{1-\lambda} \sum_{j=1}^n ||\bx_j(0)||_1  + \frac{pn \max_j B_j}{1-\lambda} (1 + \ln (\tau-1) \right) 
\cr & + \frac{p q^2 }{\tau }.\end{align} 
Finally, by convexity we have \aor{for each $i=1, \ldots, n$,} \begin{small}  
\[ \frac{\sum_{t=1}^{\tau-1} t \left( F( \bz_i(t+1) ) - F(\bz^*) + \sum_{j=1}^n \mu_j \|\bz_j(t+1) - \bz^*\|^2 \right)}{\tau(\tau-1)\aor{/2}}  \] 
\begin{equation} \geq  
F (\widehat{\bz}_i(\tau)) - F(\bz^*) + \sum_{j=1}^n \mu_j \|\widehat{\bz}_j(\tau) - \bz^*\|^2\label{eq:convexity} \end{equation} \end{small}
\aor{Putting together} Eqs.~(\ref{eq:main7}) and~(\ref{eq:convexity}) concludes the proof. 
 \end{proof}

%%%%%%%%%%%%%%%%%%%%%%%%%%%%%%%%%
\section{Proof of Theorem \ref{thm0}}\label{sec:proof0}
%%%%%%%%%%%%%%%%%%%%%%%%%%%%%%%%%

We begin by briefly sketching the main idea of the proof. The proof proceeds by simply arguing that if
$\max_i \|\bz_i(t)\|$ \an{gets} large, it decreases. 
Since the stochastic subgradient-push protocol (Eq.~\eqref{eq:minmet}) is somewhat
involved, proving this will require some involved arguments relying on the level-set boundedness of strongly convex functions with Lipschitz gradients and some
special properties of element-wise ratios of products of column-stochastic matrices.

Our starting point is a lemma that exploits the structure of strongly convex functions with Lipschitz gradients. 
\begin{lemma} Let $q:\R^d\to\R$ be a $\mu$-strongly convex 
function with $\mu>0$ and have Lipschitz continuous gradients with constant $M > 0$. 
Let $v\in\R^d$ and let $u\in\R^d$ be defined by
\[ u = v - \alpha \left(\nabla q(v) +\phi(v)\right), \]
where \aor{$\alpha\in(0,\frac{\mu}{8M^2}]$} and $\phi:\R^d\to\R^d$ is a mapping such that 
\aor{\[\|\phi(v)\|\le c\qquad\hbox{for all }v\in\R^d.\]}
Then, there exists a compact set ${\cal V}\subset\R^d$ \aor{ (which depends on $c$ and the funtion $q(\cdot)$ but not
on $\alpha$)}  such that 
\[\|u\|\leq\left\{ \begin{array}{ll}
\|v\| & \hbox{for all $v\not\in {\cal V}$}\cr
R & \hbox{for all $v\in {\cal V}$},
\end{array}\right.\]
where \aor{$R=\max_{z\in {\cal V}}\left\{\|z\| + (\mu/(8M^2)) \|\nabla q(z)\|\right\}+( \mu c)/(8M^2)$.}
%\hbox{for all $\alpha \in (0,\hat\alpha]$ and $\nu$ with $\|\nu\|\le c$}.\]
\label{normdec}
\end{lemma}

\begin{proof} \aor{The strong convexity of the function $q$ implies  
\[ \nabla q(v)' v \geq q(v) - q(0) + \frac{\mu}{2} ||v||^2 \] Consequently,} for the vector $u$ we have \begin{small}
\begin{align}
\|u\|^2 
& =  \|v\|^2 - 2 \alpha \aor{ (\nabla  q(v) + \phi(v))}' v + \alpha^2 \| \nabla q(v) +\phi(v)\|^2 \nonumber \\ 
%& =  \|v\|^2 - 2 \alpha \nabla q'(v) (v - 0) + \alpha^2 \| \nabla q(v) +\nu\|^2 \nonumber \\
& \leq  (1-\alpha\mu) \|v\|^2 - 2 \alpha (q(v) - q(0)) \nonumber \cr
&\quad \aor{- 2 \alpha \phi(v)' v} + \alpha^2 \| \nabla q(v) +\phi(v)\|^2,
\end{align}  \end{small}
For the last term in the preceding relation, 
we write
\begin{align}\label{eq:jedan}
\|\nabla q(v) +\phi(v)\|^2  \le 2\|\nabla q(v)\|^2 +2\|\phi(v)\|^2,\end{align}
where we use the  inequality
\begin{align}\label{eq:simple}
(a+b)^2 \le 2(a^2 + b^2)\qquad\hbox{for any $a,b\in\R$}. 
\end{align}
We can further write
\begin{align}\label{eq:dva}
\|\nabla q(v)\|^2
&\le (\|\nabla q(v) - \nabla q(0)\| + \|\nabla q(0)\|)^2 \cr
&\le 2M^2 \|v\|^2 + 2\|\nabla q(0)\|^2,
\end{align}
where the last inequality is obtained by using Eq.~\eqref{eq:simple} and by exploiting
the Lipchitz property of the gradient of $q$.
Similarly, using the given growth-property of $\|\phi(v)\|$  we obtain
\begin{align}\label{eq:tri}
\aor{\|\phi(v)\|^2 \leq c^2, ~~~\aor{|\phi(v)' v|}  \le \aor{ c \|v\|}. }
\end{align}
By substituting Eqs.~\eqref{eq:dva}--\eqref{eq:tri} in relation~\eqref{eq:jedan},
we find 
\[\|\nabla q(v) +\phi(v)\|^2  \le \aor{4M^2} \|v\|^2 +4\|\nabla q(0)\|^2 + 4c^2.\]
\aor{Therefore,}
\begin{align*} 
\|u\|^2 
& \leq  \left(1-\alpha(\mu-\aor{4 \alpha M^2})\right)\|v\|^2 \cr
& \quad - 2 \alpha (q(v) - q(0)) \aor{+2 \alpha c \|v\|} + 4\alpha^2 ( \|\nabla q(0)\|^2 +c^2),
\end{align*}
which for $\alpha\in \aor{(0,\frac{\mu}{8M^2}]}$ yields 
\an{
\begin{align}\label{eq:lset}
 \|u\|^2 \leq 
 &\|v\|^2 - \alpha \left( \frac{\ \mu}{2} \|v\|^2 - 2 c \|v\| \right) \cr
& - 2 \alpha \left( q(v) - q(0)  + 2 \alpha ( \|\nabla q(0)\|^2 +c^2) \right).\end{align}
} 
Define the set ${\cal V}'$ to be the following level set of $q$:
\[{\cal V}'=\{z \mid q(z)  \leq q(0) + 2 \aor{\frac{\mu}{8M^2}} (\|\nabla q(0)\|^2 +c^2 )\}.\]
Being the level-set of a strongly-convex function, the set ${\cal V}'$ is compact~\cite{BNO}(see Proposition 2.3.1(b), page 93).
Let $B({\bf 0}, \e)$ be the Euclidean ball centered at the origin and with a radius $\e>0$. Define the set ${\cal V}$ as follows:
\[{\cal V} = {\cal V}' \cup B({\bf 0}, 4c/\mu).\]

\an{If $v$ is such that $\|v\| \geq 4c/ \mu$ and $q(v)  \geq q(0) + 2\alpha (\|\nabla q(0)\|^2 +c^2 )$ (i.e., $v\not\in {\cal V}$),
then by relation~\eqref{eq:lset},} we obtain $\|u\|^2\le \|v\|^2.$
On the other hand, 
if $v\in {\cal V}$, then by using the definition of $u$ %, the Lipschitz-gradient property of $q$
and the bound $||\phi(v)|| \leq c$ we can see that
\[\|u\|\le \aor{\|v\|} +\alpha\|\nabla q(v)\|+\alpha c.\]
By using the upper bound on $\alpha$ 
we obtain the stated relation. % for this case.
\end{proof}

We next state an important relation for the images of two vectors under a linear 
transformation with a column-stochastic matrix. \aor{This is a generalization of 
a relation from  \cite{Bertsekas:1997}, Section 7.3.2.}

\begin{lemma}\label{lem:column-stoch}
Suppose $P$ is an $n\times n$ column-stochastic matrix with positive diagonal entries, 
and let $u,v\in\R^n$ with the vector $v$ having all entries positive. 
Consider the vectors $\hat u$  and $\hat v$ given, respectively, by
\begin{eqnarray*} 
\hat u = P u,\qquad
\hat v = P v.
\end{eqnarray*} 
Define the vectors $r$ and $\hat r$ with their $i$'th entries given by
$r_i=u_i/v_i $ and $\hat r_i = \hat u_i/\hat v_i$, respectively.
Then, we have
 \[ \hat r = Qr, \] 
 where $Q$ is a row-stochastic matrix.  \label{contraction}
\end{lemma}

\begin{proof} 
Indeed,  note that 
\[ \hat u_i = \sum_{j=1}^n P_{ij} u_j \qquad\hbox{for all }i. \] 
Since $\hat u_i=\hat v_i \hat u_i$  and $u_j=v_j r_j$, 
the preceding equation can be rewritten as 
\[ \hat v_i \hat r_i = \sum_{j=1}^n P_{ij} v_j r_j.\] 
Since $v$ has all entries positive and $P$ has positive diagonal entries, it follows that 
$\hat v$ also has all entries positive. Therefore, 
\[ \hat r_i = \frac{1}{\hat v_i} \sum_{j=1}^n P_{ij} v_j r_j
=\sum_{j=1}^n \frac{P_{ij} v_j}{\hat v_i} r_j. \] 
Define the matrix $Q$ from this equation, i.e.,
$Q_{ij}=\frac{P_{ij} v_j}{\hat v_i}$ for all $i,j$. The fact that $Q$ is row-stochastic
follows from $\hat v = Pv$.
%\[\sum_j=1}^n Q_{ij}=\sum_{j=1}^n \frac{P_{ij} v_j}{\hat v_i} =\frac{\hat v_i}{\hat v_i}=1.\]
\end{proof}

\smallskip

\aor{Informally speaking, the above lemma reduces the ``push-sum'' iteration to a simple stochastic
``consensus'' update. We note that it could be used to provide considerable simplifications of many of the 
arguments that have been used to show the convergence of push-sum in the past, though this
is beyond the scope of the present paper. With this lemma in place, we now proceed to prove
our second theorem.} \\
\begin{proof}[Proof of Theorem~\ref{thm0}] 
Letting $y(t)$ be the vector with entries $y_i(t)$, we can write
$y(t+1)=A(t)y(t),$ where $A(t)$ is the matrix given in Eq.~\eqref{eq:defA}.
Thus, since $y_i(0)=1$ for all $i$, we have
\[y(t)=A(t)A(t-1)\cdots A(0)\1\qquad\hbox{for all $i$ and $t\ge1$},\]
where $\1$ is the vector with all entries equal to 1.
Under Assumption~\ref{assume:main1}(a), we have shown in~\cite{AA2013} (see there Corollary 2(b))
that for all $i$,
\[\delta=\inf_{t=0,1,\ldots} \left( \min_{1\le i\le n}[A(t)A(t-1)\cdots A(0)\1]_i\right) >0.\]
Therefore, we have 
\begin{align}\label{eq:yb}
y_i(t)\ge\delta\qquad\hbox{for all $i$ and $t$}.
\end{align}
Thus, using the definition of $\bx_i(t+1)$, we can see \an{that} for all $t \geq 1$, 
\begin{eqnarray*} \bx_i(t) & = & \aor{\bw_i(t) - \alpha(t) \bg_i(t)} \\ 
& = & y_i(t) \left( \bz_i(t) - \frac{\alpha(t)}{y_i(t)} \bg_i(t) \right)
\end{eqnarray*}
implying that for all $i$ and $t\ge1$,
\begin{align}\label{eq:imp}
\frac{\bx_i(t)}{y_i(t)}= \bz_i(t) - \frac{\alpha(t)}{y_i(t)}\bg_i(t).
\end{align}

Since the matrix $A(t)A(t-1)\cdots A(0)$ is column stochastic and $y(0)=\1$, we have that $\sum_{i=1}^n y_i(t)=n$. Therefore,
$y_i(t)\le n$, which together with Eq.~\eqref{eq:yb} and $\alpha(t)\to0$ yields
\[\lim_{t\to0}\frac{\alpha(t)}{y_i(t)}=0\qquad\hbox{for all } i.\] 
Therefore, for every $i$, there is a time $\tau_i>1$ such that 
$\alpha(t)/y_i(t)\le \frac{\mu_i}{\aor{8 M_i^2}}$ for all $t\ge \tau_i$.
Hence, for each~$i$, Lemma~\ref{normdec} applies to 
the vector $\bx_i(t)/y_i(t)$ for $t\ge\tau_i$. 
By Lemma~\ref{normdec}, it follows that for each function $f_i$,
there is a compact set ${\cal V}_i$ and a time $\tau_i$ such that for all $t\ge \tau_i$,
\begin{align}\label{eq:bounded-xy}
\left\|\frac{\bx_i(t)}{y_i(t)}\right\|\le \left\{\begin{array}{ll}
\|\bz_i(t)\| & \hbox{if $\bz_i(t)\not\in {\cal V}_i$},\cr
R_i & \hbox{if $\bz_i(t)\in {\cal V}_i$},\end{array}\right.
\end{align}

Let $\tau=\max_i{\tau_i}$.
By using the mathematical induction, we will prove that for all $t \geq \tau$, 
\begin{align}\label{eq:time}
\max_{1\le i\le n}\|\bz_i(t)\|\le \bar R,
\end{align}
where $\bar R=\max\{\max_i R_i,\max_j\|\bz_j(\tau)\|\}$.
Indeed, relation~\eqref{eq:time} is true for $t=\tau$. 
Suppose it is true at some time $t\aor{\geq}\tau$. 
Then, by Eq.~\eqref{eq:bounded-xy}
we have
\begin{align}\label{eq:hypo}
\left\|\frac{\bx_i(t)}{y_i(t)}\right\|\le \max\{R_i,\max_j\|\bz_j(t)\|\}%\le \max\{R_i,\bar R\}=
\le \bar R \ \hbox{for all $i$},\end{align}
where the last inequality follows by the induction hypothesis.  
Next, we use Lemma~\ref{contraction} with $v=y(t)$, $P=A(t)$, and $u$ taken as the vector of the $\ell$'th
coordinates  of the vectors $\bx_j(t)$, $j=1,\ldots,n$, where the coordinate index $\ell$ is arbitrary.
In this way, we obtain that each vector $\bz_i(t+1)$ is a convex combination of the vectors $\bx_i(t)/y_i(t)$,
i.e., %for all $t\ge0$,
\begin{align}\label{eq:imp1}
\bz_i(t+1)=\sum_{j=1}^n Q_{ij}(t)  \frac{\bx_j(t)}{y_j(t)} \ \hbox{for all $i$ and $t\ge0$},
\end{align}
where $Q(t)$ is a row stochastic matrix with entries $Q_{ij}(t)=\frac{A_{ij}(t) y_j(t)}{y_i(t+1)}$.
By the convexity of the Euclidean norm, it follows that for all $i$,
\[\|\bz_i(t+1)\|\le \sum_{j=1}^n Q_{ij}(t)  \left\|\frac{\bx_j(t)}{y_j(t)}\right\|\le 
\max_{1\le j\le n}\left\|\frac{\bx_j(t)}{y_j(t)}\right\|,\]
which together with Eq.~\eqref{eq:hypo} yields $\|\bz_i(t+1)\|\le \bar R$,
thus implying that at time $t+1$ we have
\[\max_{1\le i\le n} \|\bz_i(t+1)\|\le \bar R.\]
Hence, Eq.~\eqref{eq:time} is valid for all $t\ge\tau$.

Note that the constant $\bar R$ is random as it depends on 
the random vectors $z_i(\tau)$, $i=1,\ldots,n$, where the time $\tau$ is deterministic.
\aor{However, we next argue that we may replace $\bar R$ with a deterministic constant
which upper bounds all $\|\bz_i(t+1)\|, i = 1, \ldots, n, t \geq 1$. Indeed,  it would suffice to find a constant which 
upper bounds all $\|z_i(t)\|$ with $i  = 1, \ldots, n$ and $t=1, \ldots, \tau$. }

Using Eqs.~\eqref{eq:imp} and~\eqref{eq:imp1}, we can see that for all $t\ge1$,
\[\max_{1\le i\le n}\|\bz_i(t+1)\|
\le \max_{1\le j\le n}\left(\|\bz_j(t)\|+\frac{\alpha(t)}{y_j(t)}\|\bg_j(t)\|\right).\]
\an{Let $\bz_j^*$ be the minimizer of $f_j$, which exists and is unique due to the strong convexity of 
$f_j$. Then,  by Assumption 2 it follows that}
\begin{eqnarray*} \| \bg_j(t) \| & \leq & \| \nabla f_j(\bz_j(t))\| +   c_j \cr 
& \leq & M_j \|\bz_j(t) - \bz_j^*\| +  c_j \cr
& \leq & M_j \|\bz_j(t)\| + M_j \| \bz_j^*\| + c_j.
\end{eqnarray*} 
\aor{Consequently, 
\[\max_{1\le i\le n}\|\bz_i(t+1)\|\le \gamma_1 \max_{1\le j\le n}\|\bz_j(t)\|+\gamma_2 \max_{1 \leq j \leq n} \| \bz_j^* \| + \gamma_3,\]}
where \aor{$\gamma_1=1+ (\bar \alpha/\delta)\max_j M_j$, $\gamma_2 = (\bar \alpha/\delta) \max_j M_j$,
$\gamma_3=(\bar \alpha/\delta) \max_j c_j$}, and $\bar \alpha=\max_{t}\alpha(t)$.
Thus, using the preceding relation recursively for $t=1,\ldots,\tau-1,$ and the fact that the initial points 
$\bx_i(0)$ are deterministic, we conclude that there exists a uniform deterministic bound on $\|\bz_i(t)\|$
for all $t\ge 1$ and $i$.
\end{proof} 

\aor{\noindent {\bf Remark}: It is possible to generalize our results to more general models of noise where
\[ \|N_i({\bf u})\| \leq \epsilon_i \|{\bf u}\| + c_i \] 
and each $\epsilon_i$ is small compared to $\mu_i$. We omit the 
\an{details as the proofs are essentially identical to the proofs given here for the case of $\e_i=0$}.}

\section{Simulations}\label{sec:nums}
We report some simulations of the subgradient-push method which 
experimentally demonstrate its performance.  
We consider the scalar function $F(\theta) = \sum_{i=1}^n p_i (\theta - u_i)^2$ 
where $u_i$ is a variable that is known only to node $i$. 
This is a canonical problem in distributed estimation, whereby 
the nodes are attempting to measure a parameter $\widehat \theta$. 
Each node $i$ measures $u_i = \widehat \theta + w_i $, where $w_i$ are jointly Gaussian and zero mean. 
Letting $p_i$ be the inverse of the variance
of $w_i$, the maximum likelihood estimate is the minimizer $\theta^*$ of $F(\theta)$ 
($\theta^*$ is unique provided that $p_i>0$ for at least one $i$). 
Each $p_i$ is a uniformly random variable taking values between $0$ and $1$. 
The initial points $x_i(0)$ are generated as independent random variables, each with a standard 
Gaussian distribution. This setup is especially attractive since the optimal solution can be 
computed explicitly (it is a weighted average of the $u_i$) allowing us to see exactly how far from
optimality our protocol is at every stage.

The subgradient-push method is run for 200 iterations with the stepsize
$\a(t)=p/t$ and $p=2n/(\sum_{i=1}^np_i)$.
The graph sequence is constructed over 1000 nodes with a random connectivity pattern.

Figure~1 shows the results obtained for simple random graphs where every node has two out-neighbors, 
one belonging to a fixed cycle and the other one chosen uniformly at random at each step.
The top plot shows how $\ln(|\widetilde{z_i}(t) - \theta^*|)$ decays on average 
(over 25 Monte Carlo simulations) for five randomly selected nodes.
The bottom plot shows a sample of $\ln(|\widetilde{z_i}(t) - \theta^*|)$ for a single Monte Carlo run
and the same selection of 5 nodes.

Figure 2 illustrates the same quantities for the sequence of graphs which alternate between 
two (undirected) star graphs. 
\begin{figure}[h!]
\vskip -7pc
\centering
\includegraphics[scale=0.37]{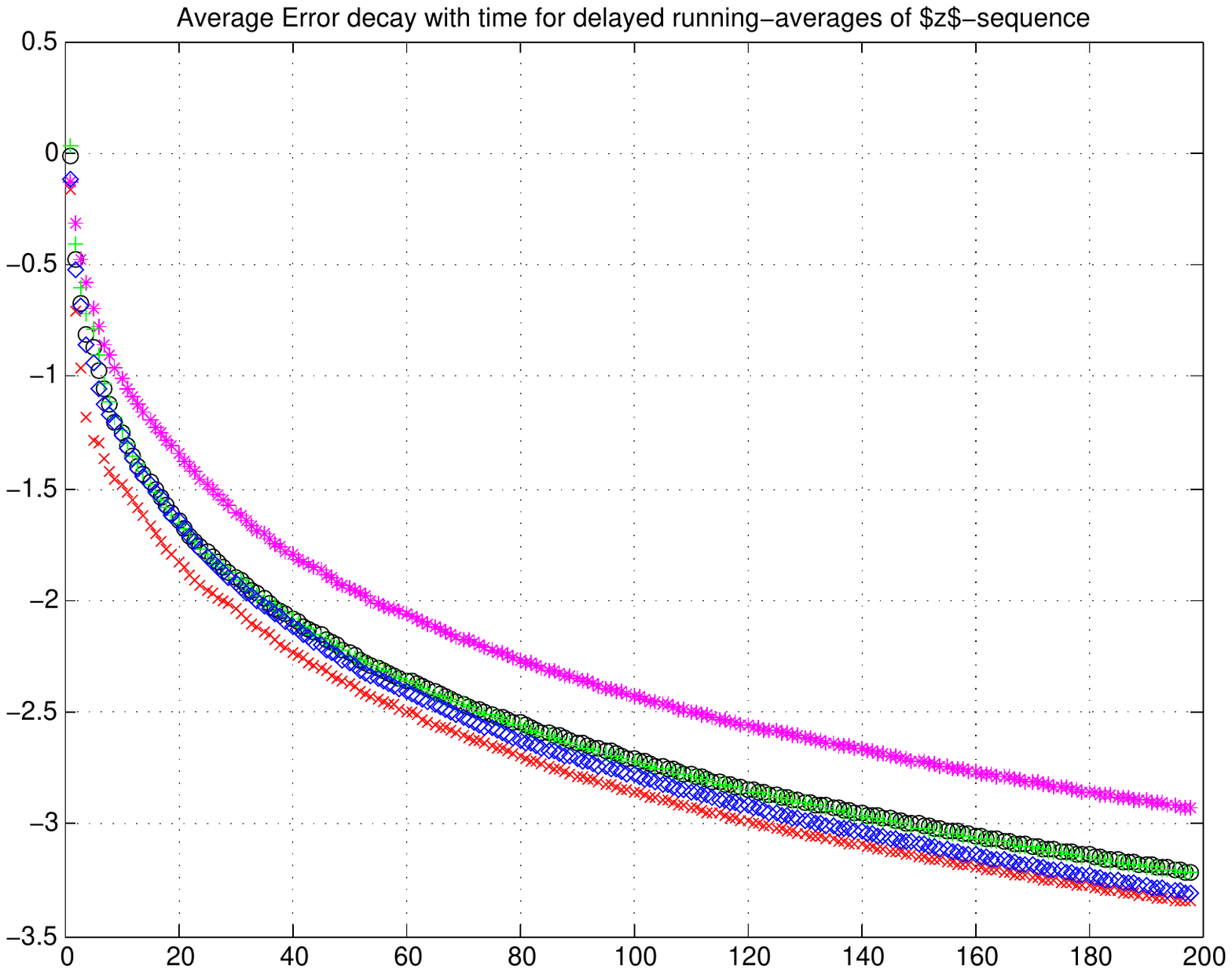}
\vskip -10pc
\includegraphics[scale=0.37]{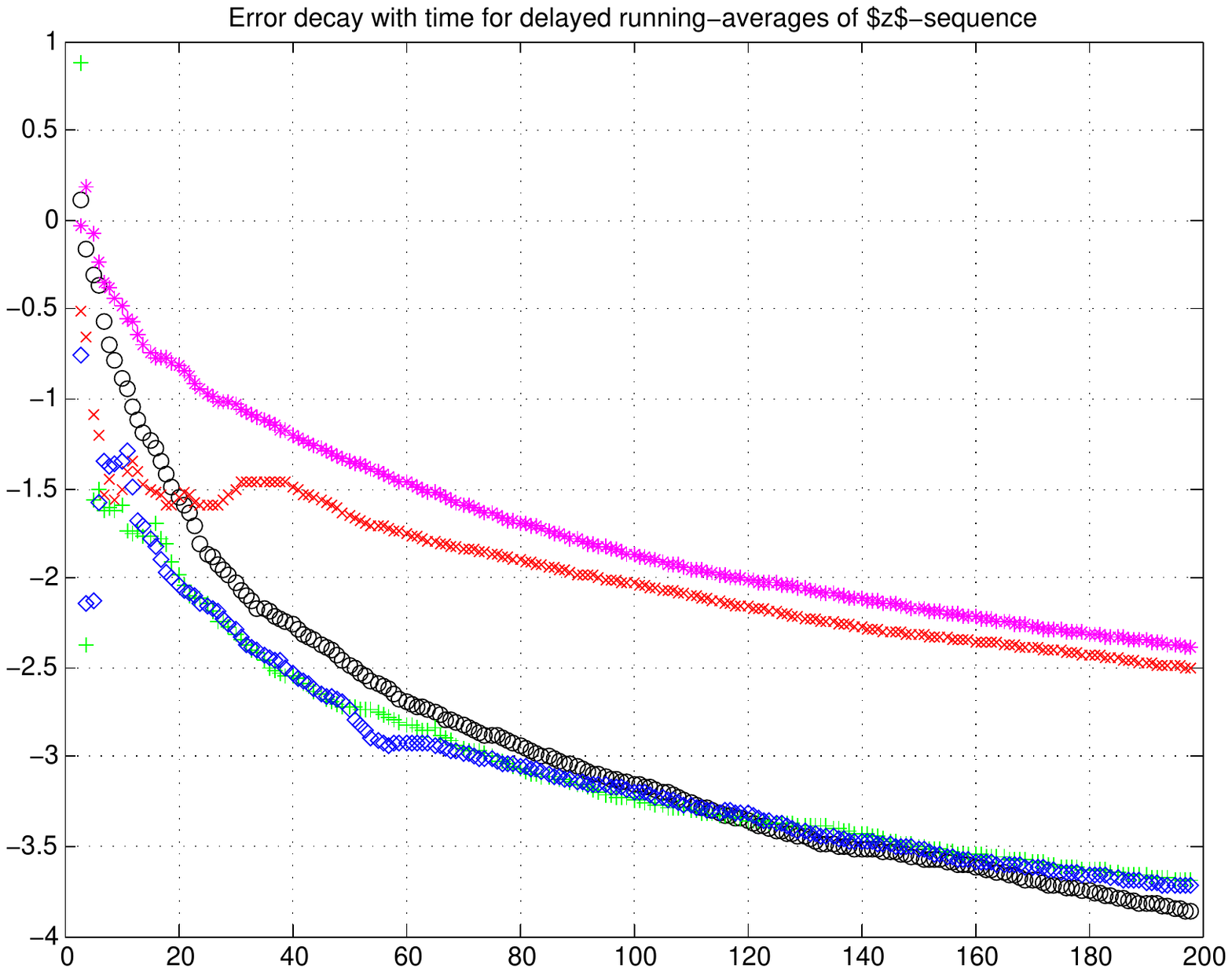}
\vskip -5pc\caption{Top plot: the number of iterations ($x$-axis) and the average of $\ln |\widetilde{z_i}(t) - \theta^*|$
($y$-axis) over 25 Monte Carlo runs for 5 randomly chosen nodes. 
Bottom plot: a sample of one a single run for the same node selection.}
\label{f1}
\end{figure} 

\begin{figure}[h!]
\vskip -4pc
\centering
{\includegraphics[scale=0.37]{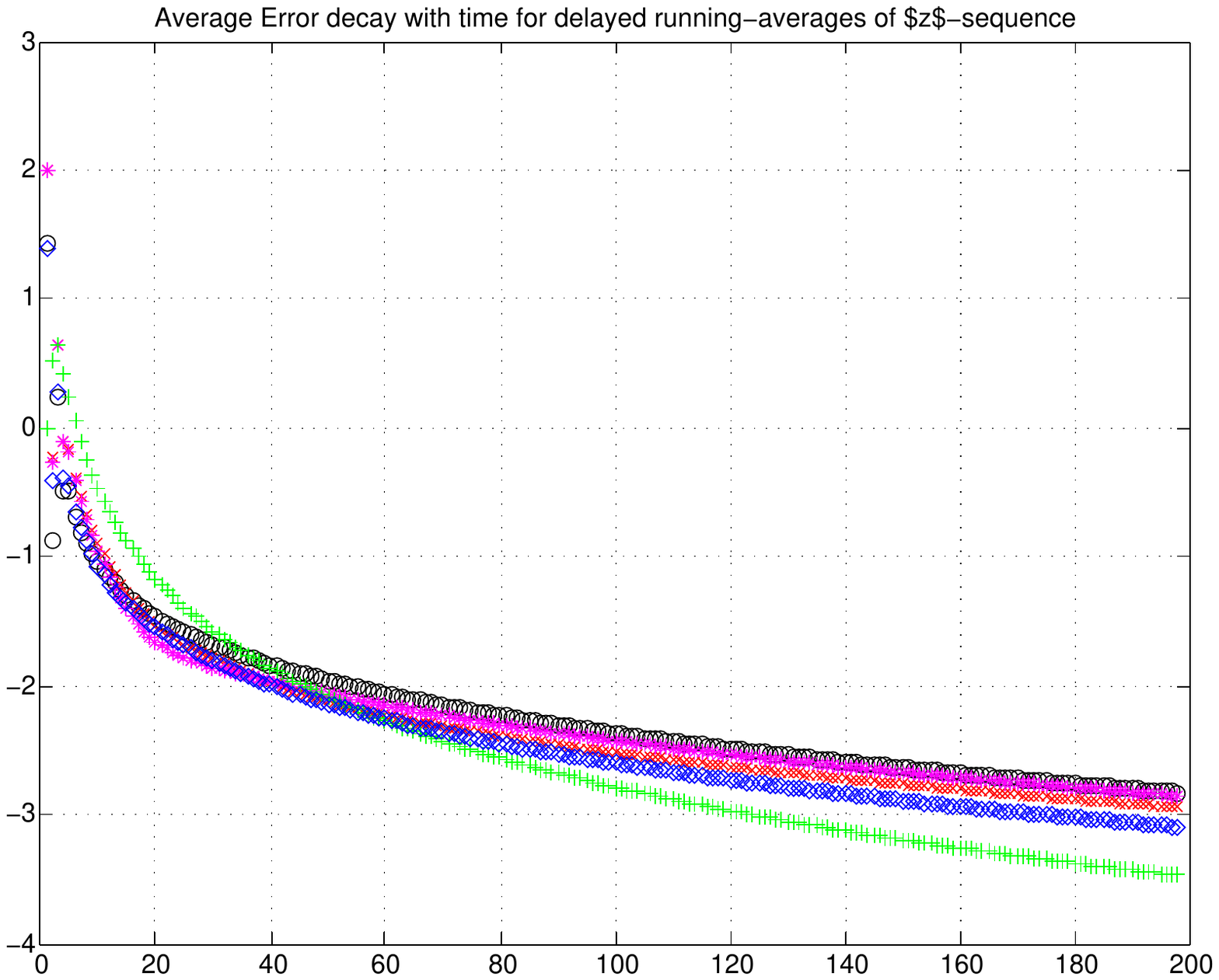}}
\vskip -10pc
{\includegraphics[scale=0.37]{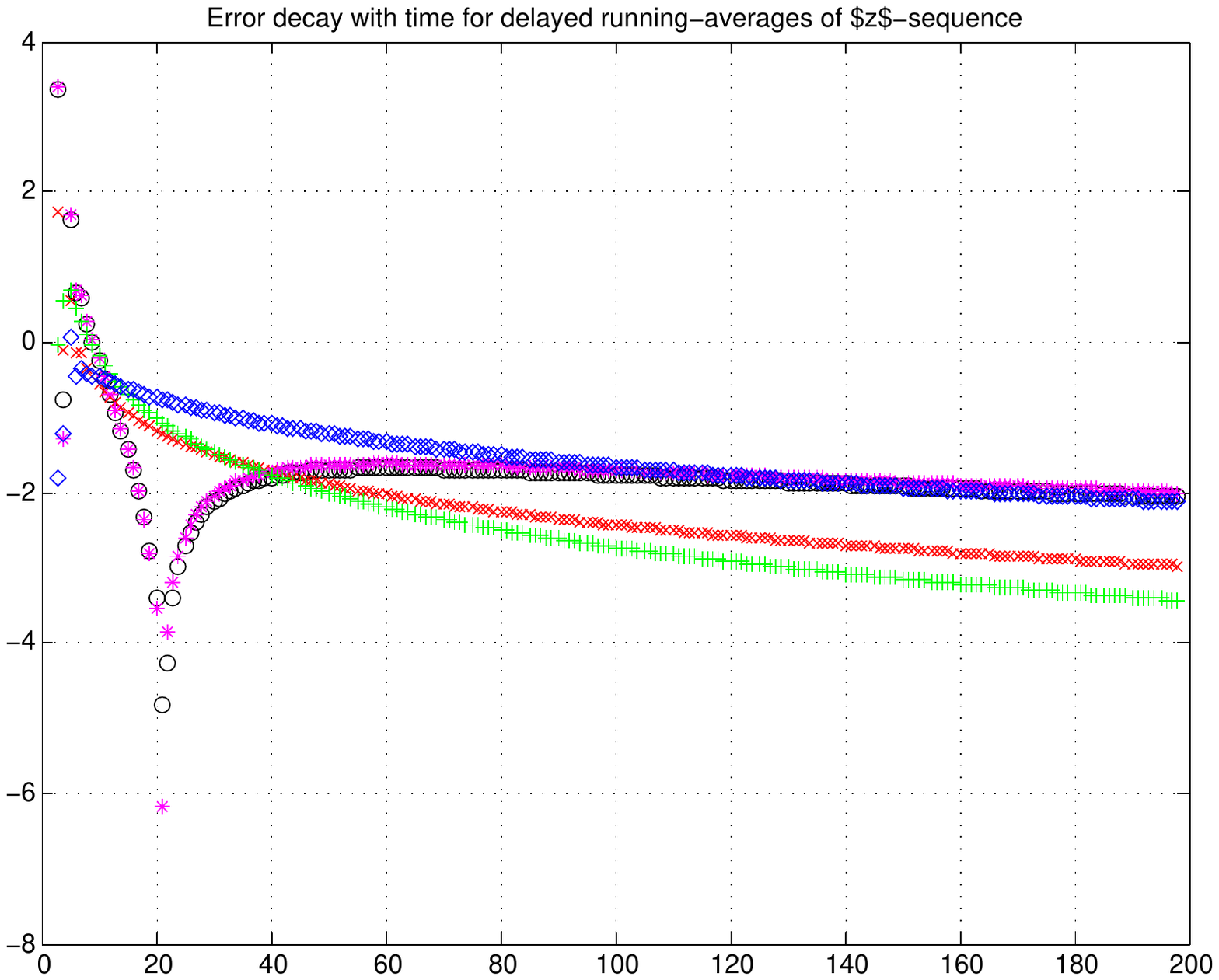}}
\vskip -5pc
\caption{Top plot: the number of iterations ($x$-axis) and the average of $\ln |\widetilde{z_i}(t) - \theta^*|$
($y$-axis) over 25 Monte Carlo runs for 5 randomly chosen nodes. 
Bottom plot: a sample of a single run for the same node selection.}
\end{figure}

We see that the error decays at a fairly speedy rate, especially given both the relatively large number of 
nodes in the system (a thousand) and the sparsity of the graph at each stage (every node has two out-neighbors).
Our simulation results suggest  the gradient-push methods we have proposed have the potential to be
effective tools for network optimization problems. For example, the simulation of Figure 1 shows that a relatively fast
convergence time can be obtained if each node can support only a single long-distance out-link.

 \section{Conclusion}\label{sec:concl}
 We have considered a variant of the subgradient-push method of our prior work~\cite{AA2013},
 where the nodes have access to noisy subgradients of their individual objective functions $f_i$. 
Our main result was that the functions $f_i$ are strongly convex functions with Lipchitz gradients, we have established 
 $O(\ln t/t)$ convergence rate of the method, which is an improvement of the previously known 
 rate $O(\ln t/\sqrt{t})$ for (noiseless) subgradient-push method shown in~\cite{AA2013}.
 %We note that the boundedness condition on the gradient noise in relation~\eqref{eq:noise} can be required to hold
 %with probability~1 instead. In this case, following our analysis with minor adjustments, one can see that 
 %Theorems 1 and~2 will also hold with probability~1.

Our work suggests a number of open questions. Our bounds on the performance of the (sub)gradient-push directly involve the convergence speed 
$\lambda$ of consensus on directed graphs. Thus the problem of designing well-performing consensus algorithms is further motivated by this work. In particular, a directed average consensus algorithm with polynomial scaling with $n$ on arbitrary time-varying graphs would lead to polynomial convergence-time scalings for distributed optimization over time-varying directed graphs. However, such an algorithm is not available to the best of the authors' knowledge.

Moreover, it would be interesting to relate the convergence speed of distributed optimization procedures to the properties possessed by the individual functions. We have begun on this research program here by showing an improved rate for strongly convex functions with Lipschitz gradients. However, one might expect that stronger results might be available under additional assumptions. It is not clear, for example, under what conditions a geometric rate can be achieved when graphs are directed and time-varying, if at all.

\aor{Finally, in many applications convergence speed should be measured not by the number of iterations but by different metrics. 
For example, it may be appropriate to count the number of bits that have to be exchanged before all nodes are close to the solution. Alternatively, 
when some of the variables correspond to physical positions which must be adjusted as a result of the protocol, the dominating
factor may be the total distance traveled by each node. Furthermore, there may be tradeoffs between these metrics that 
we do not at present understand. Understanding the performance of protocols for convex optimization in these scenarios remains an open problem.} 

\bibliographystyle{plain}
\bibliography{directed-opt}

\begin{thebibliography}{10}

\bibitem{ABRW}
A~Agarwal, P.~L. Bartlett, P.~Ravikumar, and M.~J. Wainwright.
\newblock Information-theoretic lower bounds on the oracle complexity of
  stochastic convex optimization.
\newblock {\em IEEE Transactions on Information Theory}, 58(5):pp. 3235--3249,
  2012.

\bibitem{benezit}
F.~Benezit, V.~Blondel, P.~Thiran, J.~Tsitsiklis, and M.~Vetterli.
\newblock Weighted gossip: distributed averaging using non-doubly stochastic
  matrices.
\newblock In {\em Proceedings of the 2010 IEEE International Symposium on
  Information Theory}, Jun. 2010.

\bibitem{Bertsekas:1997}
D.~P. Bertsekas and Tsitsiklis~J. N.
\newblock {\em Parallel and Distributed Computation: Numerical Methods}.
\newblock Athena Scientific, Belmont, 1997.

\bibitem{BNO}
D.P. Bertsekas, A.~Nedi\'{c}, and A.E. Ozdaglar.
\newblock {\em Convex analysis and optimization}.
\newblock Athena Scientific, 2003.

\bibitem{ChenSayed2012}
J.~Chen and A.~H. Sayed.
\newblock Diffusion adaptation strategies for distributed optimization and
  learning over networks.
\newblock {\em IEEE Transactions on Signal Processing}, 60(8):4289--4305, 2012.

\bibitem{hadj1}
A.~D. Dominguez-Garcia and C.N. Hadjicostis.
\newblock Distributed matrix scaling and application to average consensus on
  directed graphs.
\newblock {\em IEEE Transactions on Automatic Control}, 58(3):667--681, 2013.

\bibitem{dominguez}
A.D. Dominguez-Garcia and C.~Hadjicostis.
\newblock Distributed strategies for average consensus in directed graphs.
\newblock In {\em Proceedings of the IEEE Conference on Decision and Control},
  Dec 2011.

\bibitem{hadj3}
A.D. Dominguez-Garcia, C.N. Hadjicostis, and N.F. Vaidya.
\newblock Reselient networked control of distributed energy resources.
\newblock {\em IEEE Journal on Selected Areas in Communications},
  30(6):1137--1148, 2012.

\bibitem{Duchi2012}
J.C. Duchi, A.~Agarwal, and M.J. Wainwright.
\newblock Dual averaging for distributed optimization: Convergence analysis and
  network scaling.
\newblock {\em IEEE Transactions on Automatic Control}, 57(3):592 --606, 2012.

\bibitem{gharesi3}
B.~Gharesifard and J.~Cortes.
\newblock Distributed strategies for making a digraph weight-balanced.
\newblock In {\em Proceedings of the 47th Annual Allerton Conference on
  Communication, Control, and Computing}, pages 771--777, 2009.

\bibitem{gharesi1}
B.~Gharesifard and J.~Cortes.
\newblock Distributed continuous-time convex optimization on weight-balanced
  digraphs.
\newblock {\em IEEE Transactions on Automatic Control}, 59(3):781--786, 2014.

\bibitem{gharesi2}
B.~Gharesifard, B.~Touri, T.~Basar, and C.~Langbort.
\newblock Distributed optimization by myopic strategic interactions and the
  price of heterogeneity.
\newblock In {\em Proceedings of the 52nd IEEE Conference on Decision and
  Control}, 2013.

\bibitem{ciblat}
F.~Iutzeler, P.~Bianchi, P.~Ciblat, and W.~Hachem.
\newblock Asynchronous distributed optimization using a randomized alternating
  direction method of multipliers.
\newblock In {\em Proceedings of the IEEE Conference on Decision and Control},
  2013.

\bibitem{johansson}
B.~Johansson.
\newblock {\em On distributed optimization in networked systems}.
\newblock PhD thesis, Royal Institute of Technology (KTH), tRITA-EE 2008:065,
  2008.

\bibitem{johan1}
B.~Johansson, T.~Kevizky, M.~Johansson, and K.H. Johansson.
\newblock Subgradient methods and consensus algorithms for solving convex
  optimization problems.
\newblock In {\em Proceedings of the IEEE Conference on Decision and Control},
  2008.

\bibitem{dobra-kempe}
D.~Kempe, A~Dobra, and J.~Gehrke.
\newblock Gossip-based computation of aggregate information.
\newblock In {\em Proceedings of the 44th Annual IEEE Symposium on Foundations
  of Computer Science}, pages 482--491, Oct. 2003.

\bibitem{li-han}
H.~Li and Z.~Han.
\newblock Competitive spectrum access in cognitive radio networks: graphical
  game and learning.
\newblock In {\em IEEE Wireless Communications and Networking Conference},
  pages 1--6, 2010.

\bibitem{Lobel2011}
I.~Lobel and A.~Ozdaglar.
\newblock Distributed subgradient methods for convex optimization over random
  networks.
\newblock {\em IEEE Transactions on Automatic Control}, 56(6):1291 --1306, June
  2011.

\bibitem{LobelOF2011}
I.~Lobel, A.~Ozdaglar, and D.~Feijer.
\newblock Distributed multi-agent optimization with state-dependent
  communication.
\newblock {\em Mathematical Programming}, 129(2):255--284, 2011.

\bibitem{LopesSayed2007}
C.~Lopes and A.H. Sayed.
\newblock Incremental adaptive strategies over distributed networks.
\newblock {\em IEEE Transactions on Signal Processing}, 55(8):4046--4077, 2007.

\bibitem{Lu2012}
J.~Lu and C.~Y. Tang.
\newblock Zero-gradient-sum algorithms for distributed convex optimization: The
  continuous-time case.
\newblock {\em IEEE Transactions on Automatic Control}, 57(9):2348--2354, 2012.

\bibitem{nedic-broadcast}
A.~Nedi\'c.
\newblock Asynchronous broadcast-based convex optimizatio over a network.
\newblock {\em IEEE Transactions on Automatic Control}, 56(6):1337--1351, 2011.

\bibitem{ANSL2013}
A.~Nedi\'c and S.~Lee.
\newblock On stochastic subgradient mirror-descent algorithm with weighted
  averaging.
\newblock {\em SIAM Journal on Optimization}, 24(1):84--107, 2014.

\bibitem{our_conf}
A.~Nedic and A.~Olshevsky.
\newblock Distributed optimization of strongly convex functions over
  time-varying graphs.
\newblock In {\em Proceedings of the IEEE Global Conference on Signal and
  Information Processing}, 2013.

\bibitem{AA2013}
A.~Nedi\'c and A.~Olshevsky.
\newblock Distributed optimization over time-varying directed graphs.
\newblock available at: http://arxiv.org/abs/1303.2289, 2013.

\bibitem{noot09}
A.~Nedi\'c, A.~Olshevsky, A.~Ozdaglar, and J.N. Tsitsiklis.
\newblock On distributed averaging algorithms and quantization effects.
\newblock {\em IEEE Transactions on Automatic Control}, 54(11):2506--2517,
  2009.

\bibitem{AN2009}
A.~Nedi\'c and A.~Ozdaglar.
\newblock Distributed subgradient methods for multi-agent optimization.
\newblock {\em IEEE Transactions on Automatic Control}, 54(1):48 --61, Jan.
  2009.

\bibitem{neglia}
G.~Neglia, G.~Reina, and S.~Alouf.
\newblock Distributed gradient optimization for epidemic routing: A preliminary
  evaluation.
\newblock In {\em IEEE Wireless Days, 2nd IFIP}, pages 1--6, 2009.

\bibitem{NY}
A.~S. Nemirovski and D.~B. Yudin.
\newblock {\em Problem complexity and method efficiency in optimization}.
\newblock John Wiley, 1983.

\bibitem{othesis}
A.~Olshevsky.
\newblock {\em Efficient information aggregation for distributed control and
  signal processing}.
\newblock PhD thesis, MIT, 2010.

\bibitem{rabbat}
M.~Rabbat and R.D. Nowak.
\newblock Distributed optimization in sensor networks.
\newblock In {\em IPSN}, pages 20--27, 2004.

\bibitem{ram_info}
S.S. Ram, V.V. Veeravalli, and A.~Nedi\'c.
\newblock Distributed non-autonomous power control through distributed convex
  optimization.
\newblock In {\em IEEE INFOCOM}, pages 3001--3005, 2009.

\bibitem{SN2011}
K.~Srivastava and A.~Nedi\'c.
\newblock Distributed asynchronous constrained stochastic optimization.
\newblock {\em IEEE J. Sel. Topics. Signal Process.}, 5(4):772--790, 2011.

\bibitem{Tsianos2013}
K.I. Tsianos.
\newblock {\em The role of the Network in Distributed Optimization Algorithms:
  Convergence Rates, Scalability, Communication / Computation Tradeoffs and
  Communication Delays}.
\newblock PhD thesis, McGill University, Dept. of Electrical and Computer
  Engineering, 2013.

\bibitem{rabbat_allerton2012}
K.I. Tsianos, S.~Lawlor, and M.G. Rabbat.
\newblock Consensus-based distributed optimization: Practical issues and
  applications in large-scale machine learning.
\newblock In {\em Proceedings of the 50th Allerton Conference on Communication,
  Control, and Computing}, 2012.

\bibitem{rabbat_cdc2012}
K.I. Tsianos, S.~Lawlor, and M.G. Rabbat.
\newblock Push-sum distributed dual averaging for convex optimization.
\newblock In {\em Proceedings of the IEEE Conference on Decision and Control},
  2012.

\bibitem{Tsianos2011}
K.I. Tsianos and M.G. Rabbat.
\newblock Distributed consensus and optimization under communication delays.
\newblock In {\em Proc. of Allerton Conference on Communication, Control, and
  Computing}, pages 974�--982, 2011.

\bibitem{Tsianos2012}
K.I. Tsianos and M.G. Rabbat.
\newblock Distributed strongly convex optimization.
\newblock In {\em Proc. of Allerton Conference on Communication, Control, and
  Computing}, 2012.

\bibitem{rabbat_2013}
K.I. Tsianos and M.G. Rabbat.
\newblock Simple iteration-optimal distributed optimization.
\newblock In {\em Proceedings of the European Conference on Signal Processing},
  2013.

\bibitem{hadj2}
N.~H. Vaidya, C.N. Hadjicostis, and A.D. Dominguez-Garcia.
\newblock Robust average consensus over packet dropping links: analysis via
  coefficients of ergodicity.
\newblock In {\em Proceedings of the 51st Annual Conference on Decision and
  Control}, pages 2761--2766, 2012.

\bibitem{ermin}
E.~Wei and A.~Ozdaglar.
\newblock On the {O}(1/k) convergence of asynchronous distributed alternating
  direction method of multipliers.
\newblock In {\em Proceedings of the IEEE Global Conference on Signal and
  Information Processing}, 2013.

\bibitem{johan2}
B.~Yang and M.~Johansson.
\newblock Distributed optimization and games: a tutorial overview.
\newblock In {\em Networked Control Systems}, 2010.

\end{thebibliography}

\end{document}